\documentclass[11pt,leqno]{article}
\usepackage{geometry}
\newgeometry{vmargin={28mm}, hmargin={35mm,35mm}}   

\usepackage{xcolor}

\usepackage[hidelinks]{hyperref}

\usepackage{amsmath,amsthm,amsfonts,amssymb,latexsym,amscd,enumerate}
\usepackage{slashed}

\usepackage{palatino}
\usepackage[mathcal]{euler}

\usepackage{xy}
\xyoption{all}

\numberwithin{equation}{section}
\swapnumbers

\newtheorem{theorem}{Theorem}[section]
\newtheorem*{theorem*}{Theorem}

\newtheorem{proposition}[theorem]{Proposition}
\newtheorem*{proposition*}{Proposition}
\newtheorem{lemma}[theorem]{Lemma}

\theoremstyle{definition}
\newtheorem{definition}[theorem]{Definition}

\newtheorem{example}[theorem]{Example}
\newtheorem{examples}[theorem]{Examples}

\newtheorem{remark}[theorem]{Remark}
\newtheorem{remarks}[theorem]{Remarks}

\numberwithin{equation}{section}

\DeclareMathOperator{\End}{End}
\DeclareMathOperator{\clifford}{Cliff}
\DeclareMathOperator{\Cliff}{Cliff}

\DeclareMathOperator{\order}{order}

\DeclareMathOperator{\CliffordOrder}{Clifford-order} 
\DeclareMathOperator{\GetzlerOrder}{Getzler-order}
\DeclareMathOperator{\CharSpec}{{Char\hspace{1.4pt}Spec}}

\DeclareMathOperator{\str}{str}
\DeclareMathOperator{\Str}{STr}
\DeclareMathOperator{\Tr}{Tr}

\newcommand{\R}{\mathbb{R}}
\newcommand{\Z}{\mathbb{Z}}
\newcommand{\N}{\mathbb{N}}
\newcommand{\C}{\mathbb{C}}

\newcommand{\sheafA}{\boldsymbol{{A}}}
\newcommand{\sheafS}{\boldsymbol{{S}}}

\begin{document}

\title{Spinors and the Tangent Groupoid}

\author{Nigel Higson and Zelin Yi}

\date{}

\maketitle

\begin{abstract} 
The purpose of this article is to study Ezra Getzler's approach to the Atiyah-Singer index theorem from the perspective of Alain Connes'  tangent groupoid.  We shall construct a ``rescaled'' spinor bundle on the tangent groupoid, define a convolution operation on its smooth, compactly supported sections, and explain how the   algebra so-obtained incorporates Getzler's symbol calculus.
\end{abstract}


\section{Introduction}

In this paper we shall investigate the relationship between Alain Connes' tangent groupoid \cite[Sec.\ II.5]{Connes94} and Ezra Getzler's approach to the Atiyah-Singer index theorem for the Dirac operator on a spin manifold \cite{Getzler83,BerlineGetzlerVergne92}.  We shall construct a variation of Connes' convolution algebra for the tangent groupoid that incorporates Getzer's rescaling of  Clifford variables. The new algebra carries a family of supertraces that smoothly vary between  operator traces and an integral of differential forms (and so in index-theoretic contexts, where the operator traces are integer-valued, the integral of differential forms actually computes the operator trace). 

Connes introduced the tangent groupoid in order to conceptualize the construction by Atiyah and Singer  \cite{AtiyahSinger68}  of the $K$-theoretic analytic index map,
\[
\operatorname{Ind}_a\colon K(T^*M)\longrightarrow \Z ,
\]
and  thereby streamline the $K$-theory proof of the Atiyah-Singer index theorem.  In contrast, Getzler's approach to the index theorem was purely local in character, and on the surface at least, quite far removed from global, $K$-theoretic considerations.  So it is an interesting problem to try to harmonize the two approaches.  
The issue  has certainly been considered by others, but as far as we are aware  little has been published  on this topic.  It is the purpose of this paper to help fill this gap. 

We should say at the outset that virtually everything that follows is implicit either in Getzler's original work or in Connes' definition of the tangent groupoid.  But a fresh looks seems to be worthwhile, especially in view of the new role that these ideas are finding in Bismut's work on the hypoelliptic Laplacian; see for example \cite{Bismut11}.  The latter was, in fact, an important motivation for us.

The tangent groupoid of a smooth manifold $M$ is, among other things, a smooth manifold $\mathbb{T}M$ equipped with a submersion 
onto  $M{\times}\R$
via the \emph{source map} that is part of  the groupoid structure. The fibers of the source  map  have  the form 
\begin{equation}
\label{eq-source-fibers}
\mathbb{T}M _{(m,\lambda )} \cong 
\begin{cases}  
M & \lambda \ne 0 \\
T_mM & \lambda = 0 .
\end{cases}
\end{equation}
So the tangent groupoid smoothly interpolates between the curved manifold $M$ and its linear tangent spaces.

Now let $D$ be a linear partial differential operator on $M$. For $m\in M$, denote by $D_m$   the constant-coefficient, linear partial differential \emph{model operator} on the tangent space $T_mM$ that is  obtained by freezing the coefficient functions for $D$ in   local coordinates   at $m$, and then dropping lower-order terms; the resulting operator is invariantly defined on the tangent space and carries the same information as the principal symbol of $D$ at $m$.  The foundation of the relationship between the tangent groupoid and partial differential operators, and eventually between the tangent groupoid and  index theory,  is the following result:
\begin{theorem*}
If $D$ has order $q$, then the operators 
\[
D_{(m,\lambda)} = 
\begin{cases}
\lambda^{q} D & \lambda\ne 0 \\
D_m & \lambda = 0
\end{cases}
\]
constitute, under the identifications \eqref{eq-source-fibers}, a smooth family of differential operators on the  source  fibers of the tangent groupoid.
\end{theorem*}

The theorem is easy to prove, as we shall recall in Section~\ref{sec-tangent-groupoid}.  In fact it is more or less incorporated into the \emph{definition} of the tangent groupoid. 

Suppose now that $M$ is a Riemannian spin manifold with spinor bundle $S$.  In Section~\ref{sec-rescaled} we shall construct from $S$   a smooth vector bundle $\mathbb{S}$ on $\mathbb{T}M$ whose restrictions to the fibers in \eqref{eq-source-fibers} are as follows:
\begin{equation}
\label{eq-S-on-source-fibers}
\mathbb{S}\vert _{ \mathbb{T}M _{(m,\lambda)}} \cong 
\begin{cases}  
S \otimes S_m^* & \lambda\ne 0 \\
\wedge ^* T_mM & \lambda = 0 .
\end{cases}
\end{equation}
The most important, and indeed defining, feature of $\mathbb{S}$, is its relation to Getzler's 
 filtration of the algebra of linear partial differential operators acting on the sections of the spinor bundle  \cite{Getzler83}.      The filtration associates to any  operator $D$ a new family of model operators  $D_{\langle m\rangle }$  on the tangent spaces $T_mM$.  These are typically \emph{not} constant-coefficient operators, and moreover they reflect the Riemannian geometry of $M$ in a rather subtle way. We shall   prove the following result.

\begin{theorem*}
 If $D$ is a linear partial differential operator on $M$, acting on the sections of $S$, and if $D$ has Getzler-order no more than $q$, then the operators 
\[
D_{(m,\lambda)} = 
\begin{cases}
\lambda^q D & \lambda \ne 0 \\
D_{\langle m\rangle } & \lambda = 0
\end{cases}
\]
constitute a smooth family of operators on the source-fibers of $\mathbb{T}M$, acting on the sections of the smooth vector bundle $\mathbb{S}$.
\end{theorem*}

Let us return to Connes' work.  He  constructs a \emph{convolution algebra} $C_c^\infty(\mathbb{T}M)$ of smooth, compactly supported, complex-valued functions on the tangent groupoid that brings the geometric object $\mathbb{T}M$ into  close contact with operator theory. 
The tangent groupoid decomposes into a family of smooth, closed subgroupoids parametrized by $\lambda \in\R$,  namely the fibers of the composite submersion
\[
\mathbb{T}M \longrightarrow M{\times}\R \longrightarrow \R .
\]
These subgroupoids are 
\begin{equation}
\label{eq-subgroupoids}
\mathbb{T}M_\lambda  \cong 
\begin{cases} 
M{\times} M & \lambda \ne 0 \\ TM & \lambda  = 0 ,
\end{cases}
\end{equation}
where $M{\times}M$ carries the pair groupoid structure and the tangent bundle  $TM$ is a made into a groupoid using the vector group structures on its fibers.  By restricting functions on $\mathbb{T}M$  to these subgroupoids,  Connes obtains algebra homomorphisms
\begin{equation}
\label{eq-eval-t-ne-0}
\varepsilon_\lambda \colon C_c^\infty (\mathbb{T}M )\longrightarrow \mathfrak{K}^\infty(L^2(M))
\end{equation}
for $\lambda \ne 0$ and
\begin{equation}
\label{eq-eval-t-eq-0}
\varepsilon_0 \colon C_c^\infty (\mathbb{T}M)\longrightarrow  C_c^\infty (TM) ,
\end{equation}
where $\mathfrak{K}^\infty(L^2(M))$ is the algebra of smoothing operators on $L^2(M)$, and $C_c^\infty (TM)$ is the fiberwise convolution algebra of smooth, compactly supported functions on the tangent bundle.  The study of these is the next step in Connes' $K$-theoretic approach to index theory---not surprisingly so since at the level of $K$-theory, the morphisms \eqref{eq-eval-t-ne-0} are related to the analytic index of elliptic operators, whereas \eqref{eq-eval-t-eq-0} is related to the symbol class in $K$-theory, and the index theorem is all about relating these two quantities. See \cite[Sec.\ II.5]{Connes94}  or \cite{Higson93}, for further details.

We shall explore similar constructions in the spinorial context, although we shall do so here at the level of supertraces rather than $K$-theory. The first step is to construct a suitable convolution algebra. In Section~\ref{sec-multiplicative} we shall prove that   the bundle $\mathbb{S}$ carries a natural  \emph{multiplicative structure}, which is to say  a smoothly varying and associative family of complex-linear maps 
\begin{equation}
\label{eq-bundle-multiplication1}
\mathbb{S}_{\gamma}\otimes \mathbb{S}_{\eta} \longrightarrow \mathbb{S}_{\gamma \circ \eta} ,
\end{equation}
among the fibers of $\mathbb{S}$, where $(\gamma, \eta) \mapsto \gamma \circ \eta$ is the tangent groupoid composition law.  Using the multiplicative structure, the space $C_c^\infty (\mathbb{T}M, \mathbb{S})$ of smooth, compactly supported sections of $\mathbb{S}$ may be given a convolution product and becomes a complex associative algebra.

Of special interest are the multiplication maps \eqref{eq-bundle-multiplication1}  in the case where $\lambda =0$, so that  $\gamma$ and $\eta$ correspond to tangent vectors $X_m$ and $Y_m$, and  the spaces $\mathbb{S}_{\gamma}$ and $ \mathbb{S}_{\eta} $ are copies of $\wedge^* T_mM$.  We shall compute  the multiplication maps in this case, as follows:

\begin{theorem*} 
When $\lambda = 0$  the morphism \eqref{eq-bundle-multiplication1} is given by the formula 
\begin{equation}
\label{eq-bundle-multiplication2}
\alpha \otimes \beta \longmapsto \alpha \wedge \beta \wedge 
	\exp \bigl(-\tfrac 12 \kappa (X_m,Y_m)\bigr ) ,
\end{equation}
where $\kappa(X_m,Y_m)$   is the Riemannian curvature   $R(X_m,Y_m)$  viewed as an element of  $\wedge^2 T_mM$.
 \end{theorem*}
 
Continuing, and   following Connes' work, we can construct  algebra homomorphisms
\begin{equation}
\label{eq-eval-t-ne-1}
\varepsilon_\lambda \colon C_c^\infty (\mathbb{T}M, \mathbb{S})\longrightarrow \mathfrak{K}^\infty(L^2(M,S))
\end{equation}
for $\lambda \ne 0$ and
\begin{equation}
\label{eq-eval-t-eq-1}
\varepsilon_0 \colon C_c^\infty (\mathbb{T}M, \mathbb{S})\longrightarrow  C_c^\infty (TM, \wedge^* TM)) ,
\end{equation}
 by restricting sections of $\mathbb{S}$ to the subgroupoids  \eqref{eq-subgroupoids}.
Here $\mathfrak{K}^\infty(L^2(M,S))$ is the algebra of smoothing operators acting on the sections of the spinor bundle, but the algebra $ C_c^\infty (TM, \wedge^*   TM))$ is more interesting.  It is the algebra of smooth, compactly supported sections of the pullback of the exterior algebra bundle of $M$ to $TM$ but with a \emph{twisted} convolution multiplication related to \eqref{eq-bundle-multiplication2}.

In Section~\ref{sec-convolution-algebra} we shall construct and analyze  our family of supertraces 
\[
\Str_\lambda \colon  C_c^\infty (\mathbb{T}M,\mathbb{S} ) \longrightarrow \C ,
\] 
parametrized by $\lambda \in \R$. 
 When $\lambda {\ne}0$ the supertrace is defined by the diagram 
\begin{equation}
\label{eq-Str-lambda-neq-0}
\xymatrix{
C_c^\infty (\mathbb{T}M,\mathbb{S} )
\ar[dr] _{\Str_\lambda} \ar[r]^{\varepsilon_\lambda} & \mathfrak{K}^\infty (L^2 (M,S)) \ar[d]^{\Str} \\
  & \C.
} 
\end{equation}
Here $\varepsilon_\lambda$ is the evaluation morphism  in \eqref{eq-eval-t-ne-1} and $\Str$ is the standard operator supertrace.   The supertrace $\Str_0$  uses the  morphism
\[
\textstyle{\int} \colon  C_c^\infty (TM, \wedge^* TM))  \longrightarrow \C
\]
 that is given by restricting a form to the zero section $M \subseteq TM$, then integrating its top-degree component over $M$ (we use the Riemmanian metric to identify the top degree component with a top-degree differential form).  We define $\Str_0$ by means of the diagram 
 \begin{equation}
\label{eq-Str-lambda-eq-0}
\xymatrix{
C_c^\infty (\mathbb{T}M,\mathbb{S} )
\ar[dr] _{\Str_0} \ar[r]^-{\varepsilon_0} & C_c^\infty (TM, \wedge^* TM)) \ar[d]^{(2/i)^{\dim(M)/2}\cdot \int} \\
  & \C ,
} 
\end{equation}
 where $\varepsilon_0$  is the evaluation morphism  in \eqref{eq-eval-t-eq-1}.

\begin{theorem*} 
If $\sigma\in C_c^\infty (\mathbb{T}M,\mathbb{S})$, then the supertraces 
$
\Str_\lambda (\sigma) \in  \C
$
vary smoothly with $\lambda \in \R$.  \end{theorem*}
 
As we have already hinted, this is a sort of ``index theorem without a Dirac operator,''  which relates the operator supertrace to differential forms, and we shall conclude the paper with a brief reminder, following Getzler,  of how the actual index theorem for the Dirac operator quickly follows from it. 

 In future work we aim to study Fr\'echet  and Banach algebra completions  of the convolution algebra $C_c^\infty (\mathbb{T}M, \mathbb{S})$ from a $K$-theoretic perspective, and also consider variations appropriate to other occurences of Getzler's rescaling method, for instance in the  work on the hypoelliptic Laplacian that we already mentioned.

\section{The Tangent Groupoid} 
\label{sec-tangent-groupoid}
In this section we  shall review the construction of the tangent groupoid and discuss its  relations with linear partial differential operators.  The tangent groupoid is a special case of the deformation to the normal cone construction from algebraic geometry, and generally speaking we shall follow the algebraic geometric approach towards its definition.  The particular adaptations needed to handle smooth manifolds as opposed to algebraic varieties are taken from    \cite{HajSaeediSadeghHigson18}. (A more direct account would be possible, see for instance \cite{Higson10}, but it would be a bit less convenient for our later purposes.)

\subsection{Deformation to the Normal Cone}

Throughout this section   $V$ will be a smooth manifold, and  $M$ will be a smoothly embedded submanifold of $V$.  We shall denote by $C^\infty (V)$ the $\R$-algebra of real-valued smooth functions on $V$.

Recall that if $p$ is a positive  integer, then  a smooth, real-valued  function    on $V$  is said to \emph{vanish to order $p$ or more on $M$} if it is locally a sum of products of $p$ or more  smooth, real-valued   functions on $V$, all of which vanish on $M$.  It will be convenient to extend this concept to nonpositive $p$: let us agree that if   $p$ is nonpositive, then \emph{every}   smooth, real-valued  function    vanishes to order $p$ or more.   We shall define the \emph{deformation space} (or \emph{deformation to the normal cone}) $\mathbb{N}_VM$ using the filtration on the  smooth, real-valued  functions by order of vanishing on $M$, as encoded in  the following Rees  \cite{Rees56} construction:

\begin{definition} 
Denote by 
$A(V,M)\subseteq C^\infty(V)[t^{-1},t]$    the $\R$-algebra of those Laurent polynomials
$$
\sum_{p\in \mathbb{Z}} f_p t^{-p}
$$
for which  each coefficient $f_p$  is a  smooth, real-valued  function on $V$ that vanishes to order $p$ or more on $M$ (and all but finitely many $f_p$ are zero).
\end{definition}

\begin{definition}
A \emph{character} of an associative algebra $A$ over $\R$ is a non-zero algebra homomorphism from $A$ to $\R$. The \emph{character spectrum} of $A$, which we shall denote by $\CharSpec(A)$, is the set of all characters of $A$. We equip $\CharSpec(A)$ with the weak topology, that is, the topology with the fewest open sets for which the evaluation morphisms $\varphi \mapsto \varphi (a)$ are continuous.
\end{definition}

In the case of $A(V,M)$,  two kinds of characters present themselves:
\begin{enumerate}[\rm (i)]

\item Simple evaluations at   points in $V$ and a nonzero values in $\R$: given by the formula
\begin{equation}
\label{eq-eval-at-point}
\varepsilon_{(v,\lambda)} \colon \sum_{p\in \mathbb{Z}} f_p t^{-p} \mapsto \sum_{p\in \mathbb{Z}} f_p(v) \lambda^{-p} .
\end{equation}

\item Evaluations at a normal vectors  $X_m\in T_mV/T_mM$    given by the formula
\begin{equation}
\label{eq-eval-at-normal-vector}
\varepsilon_{X_m} \colon
\sum_{p\in \mathbb{Z}} f_p t^{-p} \mapsto \sum_{p\ge 0} \frac{1}{p!}X_m^p(f_p) .
\end{equation}
 Here, in order to evaluate the right-hand side, the normal vector $X_m \in T_mV/ T_m M$ is first  lifted to a tangent vector on $V $, and then extended to a vector field $X$ on $V$, so that the $p$'th iterated derivative $X^p (f_p)$ can be formed; the value of $X^p (f_p)$ at $m$ depends only on the normal vector $X_p$.

\end{enumerate}

\begin{remark}
\label{rem-exponential}
If $M$ has a Riemannian structure then the evaluation \eqref{eq-eval-at-normal-vector} at a normal vector can be written alternatively as
\begin{equation}
\label{eq-geometric-exponential}
\sum_{p\in \mathbb{Z}} f_p t^{-p} \mapsto \lim_{\lambda \to 0}\sum_{p\in \Z} {f_p(\exp_m(\lambda X_m))}{\lambda^{-p}} ,
\end{equation}
which  should help explain the phrase ``evaluation at a normal vector.''   Compare also Proposition~\ref{prop-exponential-formula} below.
\end{remark}

\begin{theorem}[See for example {\cite[Sec.\ 3]{HajSaeediSadeghHigson18}}]  
\label{thm-determination-of-spectrum}
The character spectrum of the algebra $A(V,M)$ consists precisely   of the characters of the form \eqref{eq-eval-at-point} and \eqref{eq-eval-at-normal-vector}.  All of them are distinct, and so the character spectrum may be identified with the disjoint union
\[
NM{\times}\{0\}\,\,\sqcup \,\,V{\times} \mathbb{R}^{\times}.
\]
where $NM = TV\vert _M / TM$ is  the normal bundle of $M$ in $V$, and $\R^{\times} = \R \setminus \{0\}$.  \qed
\end{theorem}

\begin{definition} 
We shall denote by  
$\mathbb{N}_VM$
the above character spectrum of $A(V,M)$.
\end{definition}

The topological space $\mathbb{N}_VM$  may be equipped with a smooth manifold structure, as follows:

\begin{definition} 
\label{def-sheaf-of-smooth-functions}
We shall denote by $\sheafA_{\mathbb{N}_VM}$ the sheaf of those real-valued continuous functions on $\mathbb{N}_VM $  that are locally of the form  
\[
\CharSpec\bigr (A(V,M)\bigl )\ni \varphi \longmapsto h(\varphi(f_1),\dots , \varphi(f_k))\in \R,
\]
where $k\in \mathbb{N}$, $f_1,\dots, f_k\in A(V,M)$, and where $h$ is a smooth function on $\R^k$.
\end{definition}

\begin{theorem}[See for example {\cite[Sec.\ 3]{HajSaeediSadeghHigson18}} again.]
The deformation space  $\mathbb{N}_VM$  carries  a unique smooth manifold structure for which  $\sheafA_{\mathbb{N}_VM}$ is the sheaf of smooth, real-valued functions. \qed
\end{theorem}

Of course every element of $A(V,M)$ can be viewed as a smooth function on $\N_VM$, and from here onwards we shall occasionally refer to $A(V,M)$ as the \emph{coordinate algebra} of the deformation space, and use the more suggestive notation 
\[
A(\mathbb{N}_VM) = A(V,M).
\]
  Not every smooth function on the deformation space belongs to the coordinate algebra, since for instance every function in the coordinate algebra is polynomial in each fiber of the normal bundle. Hence the need to construct the sheaf $\sheafA_{\mathbb{N}_VM}$ above.

\subsection{Vector Fields on the Deformation Space}
\label{subsec-vector-fields-def-space}
In this subsection we shall present the characters of $A(\mathbb{N}_VM)$ that correspond to tangent vectors in a slightly different way.  It is the algebraic counterpart of the observation in Remark~\ref{rem-exponential}, and the adjusted perspective will be quite useful for us later on. We begin with the following very simple fact:
\begin{lemma} 
\label{lem-simple-order-of-vanishing}
Let $M$ be a smooth submanifold of a smooth manifold $V$ and let $X$ be a   vector field  on $V$.  If a smooth function $f$ on $V$  vanishes to order $p$ or more on $M$, then $X(f)$ vanishes to order $p{-}1$ or more. \qed\end{lemma}

\begin{proposition}
\label{prop-extended-vector-field}
Let $X$ be a vector field on $V$. The smooth family of vector fields
\[
\lambda X \colon C^\infty(V{\times}\{\lambda\})\longrightarrow C^\infty (V{\times}\{\lambda\})
\]
defined on the fibers of $\mathbb{N}_VM$ over $\lambda \ne 0$,   extends to a smooth family of vector fields on   fibers of $\mathbb{N}_VM$ over all $\lambda \in \R$ \textup{(}here by a {smooth family of vector fields} on the fibers of a submersion we mean a vector field on the total space that is tangent to each fiber\textup{)}.
 \end{proposition}

\begin{proof}
If follows from Lemma~\ref{lem-simple-order-of-vanishing}  that the formula 
\begin{equation}
\label{eq-X-derivation}
\sum f_p t^{-p} \longmapsto \sum X(f_p) t^{-(p-1)}
\end{equation}
defines a derivation of the coordinate algebra $A(\mathbb{N}_VM)$, and it follows (see \cite[\S 2]{HajSaeediSadeghHigson18}) that there is a vector field  on the deformation space that implements this derivation. Clearly it restricts to $\lambda X$ on $V{\times}\{ \lambda\}$. 
\end{proof}

\begin{definition}
\label{def-A-zero}
We shall denote by $A_0(\mathbb{N}_VM)$ the quotient of the coordinate ring $A(\mathbb{N}_VM)$ by the ideal generated by $t\in A(\mathbb{N}_VM) $.
\end{definition}

Evidently the algebra $A_0(\mathbb{N}_VM)$ is the associated graded algebra for the decreasing filtration of $C^\infty(V)$ defined by  the order of vanishing on $M$; the algebra   $A_0(\mathbb{N}_VM)$  is  also isomorphic, via the characters \eqref{eq-eval-at-normal-vector}, to the algebra of smooth functions on the normal bundle that are polynomial  in each fiber (and are of uniformly  bounded degree).

Now denote by 
\[
\varepsilon_m\colon A_0(\mathbb{N}_VM) 
\longrightarrow \R
\]
the character obtained by evaluation at the zero normal vector at $m\in M$.  Algebraically it is given by the formula 
\[
\varepsilon _m \colon \sum f_p t^{-p} \longmapsto f_0(m).
\]
In addition, if $X$ is a vector field on $V$, then denote by 
\[
\boldsymbol{X}\colon A_0(\mathbb{N}_VM) \longrightarrow A_0(\mathbb{N}_VM)
\]
the derivation induced from the derivation of $A(\mathbb{N}_VM)$ in   \eqref{eq-X-derivation}; the latter annihilates $t$, and hence descends to the quotient $A_0(\mathbb{N}_VM) $. Since $\boldsymbol{X}$ lowers the grading degree in $A_0(\mathbb{N}_VM) $ by one, it is a locally nilpotent operator. We can therefore form the exponential
\[
\exp(\boldsymbol{X}) \colon A_0(\mathbb{N}_VM) \longrightarrow A_0(\mathbb{N}_VM)
\]
using the power series.  The following is merely a reformulation of the formula \eqref{eq-eval-at-normal-vector}:

\begin{proposition}
\label{prop-exponential-formula}
Let $M$ be a smooth submanifold of a smooth manifold $V$,  let $X$ be a   vector field  on $V$.  If $f\in A_0(\mathbb{N}_VM) $ and $m\in M$, then
\begin{equation*}
\pushQED{\qed} 
\  \varepsilon_{X_m}(f) = \varepsilon_m \bigl(\exp (\boldsymbol{X})f\bigr ) .
\qedhere
\popQED
\end{equation*}
\end{proposition}

\begin{remark}
Since $\boldsymbol{X}$ is a derivation of $A_0(\mathbb{N}_VM)$, its exponential is an algebra automorphism. This makes it clear that $\varepsilon_{X_m}$ is indeed a character. 
\end{remark}

The derivations $\boldsymbol{X}$ and $\boldsymbol{Y}$ of $A_0 (\mathbb{N}_VM)$ that are associated to a pair of vector fields commute with one another. Using this, we can compute the extension of the family of vector fields $\{ \lambda X\}$ in Proposition~\ref{prop-extended-vector-field} to $\lambda =0$, as follows:
\begin{equation}
\label{eq-symbol-computation}
\begin{aligned}
\varepsilon_{Y_m}(\boldsymbol{X}f) 
	& =  \frac{d}{dt}\Big \vert _{t=0} \varepsilon_{Y_m}(\exp(t\boldsymbol{X})f)  \\
& =  \frac{d}{dt}\Big \vert _{t=0} \varepsilon_{m}(\exp( \boldsymbol{Y})\exp(t\boldsymbol{X})f)  \\
& =  \frac{d}{dt}\Big \vert _{t=0} \varepsilon_{m}(\exp( \boldsymbol{Y}+t\boldsymbol{X})f)  
 =  \frac{d}{dt}\Big \vert _{t=0} \varepsilon_{Y_m+tX_m}(f)  .\\
\end{aligned}
\end{equation}
The extension therefore acts on the fiber of the normal bundle over $m\in M$ as directional differentiation in the direction $X_m$.

\subsection{Functoriality Properties}

The deformation space is functorial in the following sense.  Given a commuting diagram
\begin{equation}
\label{eq-def-space-functoriality1}
\xymatrix{
V \ar[r]  & V' \\
M \ar[r]  \ar[u] &  M' \ar[u]
}
\end{equation}
in which the vertical maps are inclusions of  submanifolds and the horizontal maps are arbitrary smooth maps, there is an induced map
\begin{equation}
\label{eq-def-space-functoriality2}
\mathbb{N}_{V}M\longrightarrow \mathbb{N}_{V'}M' .
\end{equation}
Indeed the algebra map from $C^\infty (V') $ to $C^\infty (V)$ determined by \eqref{eq-def-space-functoriality1} determines, in turn, an algebra map 
\[
A(\mathbb{N}_{V'}M') \longrightarrow A(\mathbb{N}_{V}M) ,
\]
and so a map on character spectra in the reverse direction. In terms of the determination of the spectrum in Theorem~\ref{thm-determination-of-spectrum}, the formula for the induced map on deformation spaces is the obvious one determined by \eqref{eq-def-space-functoriality1}.
 
 Here are some properties related to functoriality that we shall use later on: 
 \begin{enumerate}[\rm (i)]
 
 \item 
If the horizontal maps in \eqref{eq-def-space-functoriality1}  are open inclusions, then so is the induced map on deformation spaces.

\item Given a diagram 
 \[
 \xymatrix{
   V_1 \ar[r]^{\Phi_1}     & W & 
  V_2 \ar[l]_{\Phi_2}  \\
   M_1 \ar[r]_{\Psi_1}  \ar[u]   & P\ar[u] & 
  M_2 \ar[l]^{\Psi_2} \ar[u]
  }
  \]
of submanifolds  in which the horizontal maps are submersions, if we form the fiber product manifolds
 \[
  V =V_1\underset{W}\times  V_2 = \{ \, (v_1,v_2) \in V_1 {\times} V_2 : \Phi_1 (v_1) = \Phi_2 (v_2) \, \}
  \]
  and 
   \[
  M =M_1\underset{P}\times M_2 = \{ \, (m_1,m_2) \in M_1 {\times} M_2 : \Psi_1 (m_1) = \Psi_2 (m_2) \, \} ,
  \]
 then the natural map 
 \[
 \mathbb{N}_V M \longrightarrow  \mathbb{N}_{V_1} {M_1} \underset{ \mathbb{N}_W P} \times   \mathbb{N}_{V_2} {M_2}
\]
coming from functoriality is a diffeomorphism.

\item If  $M=V$, that is, if the submanifold $M$ is the entire manifold $V$, then there are of course no nonzero normal vectors, and  the deformation space identifies with the product $M{\times}\R$ via   Theorem~\ref{thm-determination-of-spectrum}:
\begin{equation}
\label{eq-trivial-def-space}
\mathbb{N}_M M \cong M {\times} \R
\end{equation}

\end{enumerate}

\subsection{Tangent Groupoid}

The \emph{tangent groupoid} of a smooth manifold $M$, denoted $\mathbb{T}M$, is the deformation space associated to the diagonal embedding 
of $M$ into its square.  That is, 
\[
\mathbb{T}M = \mathbb{N}_{M^2}M ,
\]
where  $M^2 = M{\times}M$ (it will be helpful to use this compressed notation for the powers of $M$  in this subsection).  Throughout the paper we shall identify the normal bundle for the diagonal embedding with the tangent bundle of $M$ via the projection onto the first coordinate.  So as a set 
\begin{equation}
\label{eq-tangent-groupoid-described}
 \mathbb{T}M =TM{\times}\{0\} \,\sqcup  \,M^2{\times}\R^{\times}.
\end{equation}
We shall   use the following notation for the coordinate algebra of the tangent groupoid:
\[
A(\mathbb{T}M) = A ( \mathbb{N}_{M^2}M).
\]

The tangent groupoid inherits a Lie groupoid structure from the pair group\-oid structure on $M{\times}M$ using   the functoriality of the deformation space construction,  as follows.  Consider  the commuting diagram
\[
\xymatrix{
M^2 \ar@<.35ex>[r] \ar@<-.35ex>[r] & M \\
M \ar[u] \ar[r]_{=} &M  \ar[u] 
}
\]
in which the top maps are the  first and second coordinate projections---these are the target and source maps, respectively, for the pair groupoid---while the upwards maps are diagonal maps, viewed as inclusions of submanifolds.  By functoriality of the deformation space construction, the diagram  gives rise to   maps 
\[
\mathbb{N}_{M^2} M  \rightrightarrows \mathbb{N}_MM ,
\]
and therefore to maps
\[
t,s\colon \mathbb{T} M  \rightrightarrows M {\times} \R .
\]
These are the  target and source maps for the tangent groupoid.  

The composition law in the tangent groupoid  is obtained in    the same fashion.  The space of composable pairs of elements in the tangent groupoid is 
 \[
 \begin{aligned}
 \mathbb{T}M^{(2)} 
 	& = \{ \, (\gamma, \eta) \in \mathbb{T}M{\times}\mathbb{T}M : s(\gamma) = t(\eta)\,\} \\
 	& = \mathbb{T}M \underset{ {M{\times}\R}} \times \mathbb{T}M
\end{aligned}
\]
According to the previous subsection, the diagram
\begin{equation}
\label{eq-comp-tangent-groupoid0}
\xymatrix@C=40pt{
M^2   \ar[r]^{p_{2}} & M & M^2\ar[l]_{p_1} \\
M \ar[u] \ar@{=}[r] &M  \ar[u] & M \ar[u] \ar@{=}[l]
}
\end{equation}
in which the top maps are the projections onto the second and  factors respectively, gives rise to a diffeomorphism
\begin{equation}
\label{eq-tangent-groupoid-composable-pairs}
 \mathbb{N}_{M^3}M \stackrel \cong \longrightarrow \mathbb{T}M \underset{M{\times}\R}\times  \mathbb{T}M .
 \end{equation}
On the other hand the  diagram 
\[
\xymatrix{
M^3   \ar[r]& M^2 \\
M \ar[u] \ar[r]_{=} &M , \ar[u]
}
\]
in which the top map is projection onto the first and third factors induces a map
from $ \mathbb{N}_{M^3}M$   to $ \mathbb{N}_{M^2}M$ ,
and hence a  map
\begin{equation}
\label{eq-tangent-groupoid-composition}
 c\colon \mathbb{T}M  ^{(2)} \longrightarrow \mathbb{T}M .
\end{equation}
This is the composition law for the tangent groupoid.   Clearly
\begin{equation}
\label{eq-comp-tangent-groupoid1}
(m_1,m_2,\lambda) \circ (m_2,m_3,\lambda) = (m_1,m_3,\lambda)
\end{equation}
when  $\lambda {\ne}0$.  The formula when $\lambda{=}0$ is only a little harder to derive.  We shall do the calculation in a somewhat roundabout way that will be helpful later. 

Start with the following commutative diagram, in which the various  evaluation characters $\varepsilon$ are labeled by the indicated tangent vectors on the diagonal (which of course project to normal vectors):   
\begin{equation}
\label{eq-composable-tangent-groupoid-pairs}
\xymatrix@C=40pt{
A(\mathbb{N}_{M^2}M) \ar[r]^{p_{12}^*} \ar[d]_{\varepsilon_{(X_m,0_m)}}
	& A (\mathbb{N}_{M^3}M) \ar[d]_{\varepsilon_{(X_m,0_m,-Y_m)}} 
	& A(\mathbb{N}_{M^2}M) \ar[l]_{p_{23}^*}\ar[d]_{\varepsilon_{(Y_m,0_m)}}^ {= \,\,\varepsilon_{(0_m,-Y_m)}}
\\
\R \ar@{=}[r] & \R & \R  . \ar@{=}[l]
} 
\end{equation}
Commutativity follows by direct computation from the definitions. It follows from the diagram that the isomorphism 
\eqref{eq-tangent-groupoid-composable-pairs} maps the point of $\mathbb{N}_{M^3}M$ associated to the character 
$\varepsilon_{(X_m,0_m,-Y_m)}$ to 
$\bigl ( (X_m,0), (Y_m,0)\bigr )\in \mathbb{T}M^{(2)}$.

\begin{lemma}
$
(X_m,0)\circ (Y_m,0) = (X_m+ Y_m, 0)$. 
\end{lemma}

\begin{proof}Consider 
the diagram 
\begin{equation*}
\xymatrix@C=60pt{
A_0(\mathbb{T}M) \ar[r]^{c^*} 
& A_0 (\mathbb{TM}^{(2)}) \ar[r]^{\varepsilon_{(X_m,0_m,-Y_m)}} 
& \R 
\\
& A_0( \mathbb{T}M) \otimes_{\R}  A_0( \mathbb{T}M) \ar[u] \ar[r]_-{\varepsilon_{(Y_m,0_m)}\otimes \varepsilon_{(X_m,0)} } & \R \otimes_{\R} \R
\ar[u]_{\cong}
}
\end{equation*}
 in which the leftmost map is induced from  the composition law \eqref{eq-tangent-groupoid-composition}, while the left vertical map is induced from 
 \[
 f_1\otimes f_2 \longmapsto \Bigl [ (m_1,m_2,m_3) \mapsto f_1(m_1,m_2)f_2(m_2,m_3)\Bigr ] ,
 \]
 or equivalently from the inclusion of $\mathbb{T}M^{(2)}$ into $\mathbb{T}M\times \mathbb{T}M$. It follows from \eqref{eq-composable-tangent-groupoid-pairs}   that the square in the diagram is commutative.  The composition of morphisms along the top is the character of $A_0(\mathbb{T}M) $ associated to    $(X_m,0)\circ(Y_m, 0)$.  But this composition is 
\[
\varepsilon _{(X_m,-Y_m)}\colon A_0(\mathbb{T}M)\longrightarrow \R
\]
and $ \varepsilon _{(X_m,-Y_m)} = \varepsilon _{(X_m+ Y_m, 0)} $ since the tangent vectors 
$
(X_m,-Y_m)$ and $ (X_m+Y_m,0) $ at the diagonal point $(m,m)$ determine the same normal vector for the diagonal embedding. This completes the proof.
 \end{proof}

\subsection{Families of Differential Operators}

Let $D$ be a linear partial differential operator  on $M$.  The source fibers  of the pair groupoid $M{\times}M$ have the form $M{\times}\{m\}$, and if we place a copy of $D$ on each one, then we obtain a smooth, equivariant family of linear partial differential operators   on the source fibers.  

A small extension of the above produces a family of linear partial differential operators on the source fibers of the tangent groupoid.  First, if $m\in M$, then denote by $D_m$ the translation-invariant \emph{model operator} on $T_mM$ that is obtained from $D$ by freezing coefficients in a local coordinate expression for $D$ and dropping lower order terms.

\begin{theorem}
 \label{thm-families-of-ops}
Let $M$ be a smooth manifold and let $D$ be a linear partial differential operator on $M$ of order $q$.  The formula
\[
D_{(m,\lambda )} = 
\begin{cases} 
	\lambda ^q D & \lambda \ne 0 \\
	D_m & \lambda  = 0
\end{cases}
\]
defines a smooth and equivariant family of differential operators on the source fibers of $\mathbb{T}M$.
\end{theorem}

\begin{proof}
We need to show that if $f$ is a smooth function on $\mathbb{T}M$, and if we apply the above family of differential operators to $f$ fiberwise, then the result is again a smooth function on $\mathbb{T}M$.  

In fact it suffices to prove this when  $D$ is a vector field, since $\lambda ^q D$  is a sum of products 
\[
\lambda ^{q-d}\cdot h\cdot   \lambda  D_1\cdot \ldots \cdot  \lambda D_d 
\]
 where  $h$ is a smooth function,  the operators $D_i$ are vector fields, and where $d \le q$, and since the model operator $D_m$ is the sum of those  products  
 \[
  h (m)\cdot   D_{1,m}\cdot \ldots \cdot   D_{d,m} 
  \]
  where  $q=d$. The vector field case is  handled by Proposition~\ref{prop-extended-vector-field} and  \eqref{eq-symbol-computation}.
\end{proof}

\section{A Rescaled Spinor Bundle}
\label{sec-rescaled}
In this section we shall explain how to construct a ``rescaled'' spinor bundle on the tangent groupoid of a   Riemannian spin manifold.

\subsection{Clifford Algebras}
\label{sec-clifford}

We begin with a very quick review of some points in Clifford algebra theory  to fix notation and terminology.  Let $E$ be a finite-dimensional euclidean vector space.  We shall denote by  $\clifford_{\R}(E)$ the real algebra generated by a copy of $E$ subject to the relations
\[
ef + fe = -2\langle e,f\rangle 1
\]
for all $e,f\in E$, and we shall denote by $\clifford_{\C}(E)$, or simply $\clifford (E)$, its complexification. 
We shall generally follow the conventions in the monograph \cite{Meinrenken13}, and in terms of that book, $\clifford (E)$  is the complex Clifford algebra associated to the bilinear form $B$ given by the \emph{negative} of the inner product on $E$.

There is a real-linear \emph{quantization isomorphism}
\begin{equation}
\label{eq-quantization-map}
q \colon \wedge ^* E \longrightarrow \clifford_{\R} (E)
\end{equation}
as in \cite[Sec.\ 2.2.5]{Meinrenken13}.  If $\{ e_1,\dots, e_n\}$ is any orthonormal basis for $E$, then 
  \[
q(e_{i_1}\wedge \cdots \wedge e_{i_d}) = e_{i_1}\cdot \ldots \cdot e_{i_d}
\]
for all indices $i_1< \cdots < i_d$. The quantization isomorphism equips the Clifford algebra with a vector space grading. This  grading is not compatible with the multiplication operation in the Clifford algebra, but the underlying increasing filtration \emph{is} compatible with multiplication.   We shall write it   as 
\begin{equation}
\label{eq-clifford-filtration}
\C \cdot I = \clifford_{0}(E) \subseteq 
 \clifford_{1}(E)\subseteq \cdots \subseteq  \clifford_{\dim(E)}(E) = \clifford(E) ,
\end{equation}
where $\Cliff_d (E)$ is the sum of all $q(\wedge ^aE)$ with $a\le d$.

\begin{remark}
For later purposes it  will be convenient to extend this filtration to all $d\in \Z$ so that 
\begin{equation}
\label{eq-clifford-filtration2}
\clifford_{d}(E)  = 
\begin{cases} 
0 & d<0 \\
 \clifford(E) & d> \dim(E) .
\end{cases}
\end{equation}
   The associated ``Clifford order'' of $0\in \clifford (E)$ will be  $-\infty$.  Observe that the quantization map gives rise to an isomorphism 
\begin{equation}
\label{eq-clifford-filtration3}
\wedge^d E \stackrel \cong \longrightarrow \clifford_d(E) / \clifford_{d-1}(E)
\end{equation}
for all $d$.
\end{remark}

The subspace $q(\wedge^2 E)\subseteq \clifford_{\R}(E)$ is closed under the ordinary commutator bracket in the Clifford algebra, and so acquires a Lie algebra structure.  Moreover 
\[
\bigl [ q(\wedge^2 E), q(\wedge ^1 E)\bigr ] \subseteq q(\wedge ^1 E) ,
\]
so that the Lie algebra $q(\wedge ^2 E)$ acts on $E \cong q(\wedge ^1 E)$ by commutator bracket in the Clifford algebra.  This action determines a Lie algebra homomorphism
\[
q(\wedge^2 E) \longrightarrow \mathfrak{gl}(E) ,
\]
and indeed  Lie algebra  \emph{isomorphism}
\begin{equation}
\label{eq-adjoint-map}
 q(\wedge^2 E) \stackrel \cong \longrightarrow \mathfrak{so}(E).
 \end{equation}
 
 Now define a vector space isomorphism $\gamma\colon \mathfrak{so}(E) \to \wedge^2 E$ by means of   the following commuting diagram:
\begin{equation}
\label{eq-gamma-map}
 \xymatrix{
 & q(\wedge ^2 E )\ar[dr]^{\eqref{eq-adjoint-map}} & \\
 \wedge^2 E \ar[ur]^q \ar[rr]_{\gamma^{-1}}&&   \mathfrak{so}(E) . 
 }
\end{equation}
We shall not use it, but $\gamma$ is given by the beautiful  explicit formula
\[
\gamma (T) = \tfrac 14 \sum T(e_i) \wedge e_i .
\]
See \cite[Section 2.2.10]{Meinrenken13}.

\subsection{Spinor Bundles}
From now on  $M$ will be an even-dimensional,  Riemannian spin manifold. We shall     review some facts concerning   spinors on $M$.

Let $S\to M$ be a complex irreducible  spinor vector bundle, equipped with the canonical Riemannian  connection $\nabla$ (also known as the  Levi-Civita connection), as in   \cite[Sec.\  II.4]{LawsonMichelsohn89} or  \cite[Ch.\ 4]{Roe98}.  The bundle $S$ and   connection $\nabla$ have the following properties: 
\begin{enumerate}[\rm (i)]
\item  $S$ is a smooth, $\Z/2$-graded Hermitian vector bundle over $M$.

\item There is a morphism of smooth real vector bundles 
\[
c \colon TM \longrightarrow \End(S)^{\text{odd}}_{\text{\rm skew-adjoint}}
\]
with 
\[
c(X)^2 = - \|X\|\cdot I 
\]
for every vector field $X$.

\item The morphism $c$   induces an irreducible representation of $\clifford_{\C} (T_mM)$ on $S_m$ for every $m\in M$, and indeed a $\Z/2$-graded algebra isomorphism 
\begin{equation}
\label{eq-Clifford-action}
c\colon \clifford_{\C} (T_mM)\stackrel \cong \longrightarrow \End(S_m).
\end{equation}

 \item If $X$ and $Y$ are vector fields on $M$, and if $s$ is a smooth section of $S$, then
 \begin{equation}
 \label{eq-compatible-connection-on-S}
 \nabla_Y (c(X) s ) = c(\nabla_Y^{\text{\rm LC}}(X)) s + c(X)\nabla_Y s  ,
 \end{equation}
 where $\nabla^{\text{\rm LC}}$ is the Levi-Civita connection on $TM$.

\end{enumerate}
We shall also use in a crucial way  a simple formula that relates  the curvature operator 
\[
K(X,Y) = \nabla _X \nabla _Y - \nabla_Y \nabla_X - \nabla_{[X,Y]}
\]
of the Riemannian connection on $S$ to  the  Riemann curvature tensor
\[
R(X,Y) =  \nabla^{\text{\rm LC}} _X \nabla ^{\text{\rm LC}}_Y - \nabla^{\text{\rm LC}}_Y \nabla^{\text{\rm LC}}_X - \nabla^{\text{\rm LC}}_{[X,Y]}
\]
on $TM$.  For any pair of tangent vectors $X_m,Y_m\in T_mM$ we have, of course  
\[
R(X_m,Y_m)  \in \mathfrak{so}(T_mM) 
\quad \text{and} \quad 
K(X_m,Y_m)\in \End (S_m)
.
\]
Moreover
\begin{equation}
\label{eq-K-from-gamma}
 K(X,Y) =  c \circ  q\circ \gamma \bigl  (R(X ,Y)\bigr )   
\end{equation}
where $\gamma$ is the morphism \eqref{eq-gamma-map}.

\subsection{The Scaling Filtration}

Denote by  $S\boxtimes S^\ast $ the bundle over $ M{\times }M$ whose fiber over $(m_1,m_2)$ is $S_{m_1}\otimes S^*_{m_2}$.  In this subsection we shall   construct a decreasing filtration of the space of smooth sections of $S\boxtimes S^\ast $ that is based on the vanishing behavior of sections near the diagonal in $M{\times}M$. The  construction uses the following \emph{Getzler filtration} of the algebra of linear partial differential operators acting on the smooth sections of $S$ over $M$.

\begin{definition}
Let $D$ be a linear partial differential operator acting on the smooth sections of the spinor bundle $S$ over $M$.  We   say that $D$ has  \emph{Getzler order $p$ or less} if in a neighborhood of any point in $M$ it can be expressed as a finite sum of operators of the form
\[
f \cdot D_1\cdots D_p ,
\]
where $f$ is a smooth function and each $D_j$ is either a covariant derivative $\nabla_X$, or a Clifford multiplication operator $c(X)$, or the identity operator. 
\end{definition}

\begin{examples} If $X$ is any vector field on $M$, then  $\GetzlerOrder (\nabla_X) \le 1$, and the order is equal to $1$ unless $X=0$. In addition $\GetzlerOrder(c(X)) \le 1$, and again the order is equal to $1$ unless $X=0$.
\end{examples}

The construction  also uses the following increasing filtration of the fibers of $S{\boxtimes}S^*$ over the diagonal in $M{\times}M$.  Using  \eqref{eq-Clifford-action}, these fibers  admit canonical identifications 
\begin{equation}
\label{eq-canonical-identifications}
S_m^{\phantom{*}}\boxtimes S^\ast_m  \cong \End(S_m) \cong  \clifford (T_mM) ,
\end{equation}
and we equip them with the canonical increasing  Clifford algebra  filtration  from \eqref{eq-clifford-filtration}.  

\begin{definition}
Let $q\in \Z$. 
We shall say that a smooth section of  $S\boxtimes S^\ast $ has \emph{Clifford order $q$ or less}  if its value at each diagonal point $(m,m)$ lies in the  order $q$ subspace $\clifford_q(T_mM)\subseteq \clifford(T_mM)$.    
\end{definition}

\begin{examples}
If $q$ is negative, then a section with Clifford order $q$ or less at $m\in M$ must vanish. At the other extreme, every section has Clifford order $\dim(M)$ or less.
\end{examples}

In the following definition we shall consider linear partial differential operators $D$ that act on the smooth sections of the spinor bundle $S$ over $M$. We shall consider $D$ as also acting on the smooth sections of $S \boxtimes S^*$ over $M{\times}M$ by differentiation in the first factor of $M{\times} M$  alone.

\begin{definition}
\label{def-scaling-order}
Let $p\in \mathbb{Z}$. We shall say that   a section $\sigma$ of $S \boxtimes S^*$ over $M{\times}M$ has \emph{scaling order  $p$ or more}  if 
\[
\CliffordOrder (D\sigma) \le q-p
\]
for every differential operator $D$ of Getzler order $q$ or less.  If $m\in M$, then we shall say that $\sigma$ has \emph{scaling order $p$ near $m$} if the above condition holds in a neighborhood of $(m,m)$.
\end{definition}

\begin{remark}
The definition can be compared as follows to the ordinary notion of  vanishing to order $p$ along the diagonal of a real-valued function  on $M{\times}M$ that we     used to construct the tangent groupoid. Suppose we write $\operatorname{val} (f) = -\infty$ if $f$ vanishes on the diagonal, while $\operatorname{val}(f) = 0$ otherwise.  Then   
 $f$  vanishes to order $p$ on the diagonal in $M{\times} M$ if and only if 
\[
\operatorname{val}(D f) \le \order (D) - p
\]
 for every   linear partial differential operator $D$ on $M$ acting on functions on $M{\times}M$ through the first factor.  \end{remark}

\begin{example}
Since the Clifford order of a section of $S \boxtimes S^*$ is never more than $\dim(M)$, every section has scaling order $-\dim(M)$, or more.
\end{example}
 
Let us record two easy consequences of the definition.

\begin{lemma}
\label{lem-S-module-over-A}
If a smooth section $\sigma$ of $S\boxtimes S^*$ has  {scaling order}   $p_1$ or more, and if a smooth function $f$ on $M{\times}M$ vanishes to order $p_2$ or more on the diagonal of $M{\times} M$, then the section $f\cdot \sigma$ has scaling order  $p_1+ p_2$ or more. \qed
\end{lemma}

In particular, the  sections of $S \boxtimes S^*$  of scaling order $p$, or more, form a $C^\infty (M{\times} M)$-module.

\begin{lemma}
If a smooth section $\sigma$ of $S\boxtimes S^*$ has  {scaling order}   $p$ or more, and if $D$ has Getzler order $q$ or less, then $D\sigma$ has scaling order $p{-}q$ or more.  \qed
\end{lemma}

A deeper result concerning scaling order is the following fact, whose proof we shall give in an appendix; see Section~\ref{sec-taylor}.

\begin{proposition}
\label{prop-extension-from-diagonal}
Let $m\in M$ and let $d \ge 0$. Every smooth section of the bundle $\clifford (TM)$ over $M$ that has Clifford order $d$ or less near $m$  is the restriction to the diagonal in $M{\times}M$ of a smooth section of $S\boxtimes S^*$ of scaling order $-d$ or more near $m$.
\end{proposition}

\subsection{The Rescaled Spinor Module} 
\label{subsec-rescaled-module}
 
In this subsection we shall define a module $S(\mathbb{T}M)$ over the coordinate algebra $A(\mathbb{T}M)$    using   the scaling filtration from the previous subsection and the Rees construction. As we shall soon see, it  may be viewed as  the module of ``regular'' sections of a   bundle $\mathbb{S}$ over the tangent groupoid, just as $A(\mathbb{T}M)$  may be viewed as  the  algebra of ``regular'' functions on the tangent groupoid.  Here we shall compute the fibers of the module  $S(\mathbb{T}M)$, which will be the fibers of the bundle $\mathbb{S}$.

 \begin{definition}
Denote by  $S(\mathbb{T}M)$   the complex  vector space of   Laurent polynomials
\[
\sum_{p \in \mathbb{Z}} \sigma _p  t^{-p}
\]
where each $\sigma_p$ is a smooth section of $S\boxtimes S^*$ of scaling order at least $p$.   It follows from  Lemma~\ref{lem-S-module-over-A} that    $S(\mathbb{T}M)$   is a module over $A(\mathbb{T}M)$ by ordinary multiplication of   Laurent polynomials. For each point $\gamma \in \mathbb{T}M$ let   $I_\gamma\subseteq A(\mathbb{T}M)$  be  the corresponding vanishing ideal. The \emph{fiber} of $S(\mathbb{T}M)$ over $\gamma$ is 
\begin{equation*}
S(\mathbb{T}M)\vert _\gamma  = S(\mathbb{T}M)  \big/ I_\gamma \cdot S(\mathbb{T}M) 
\end{equation*}

\end{definition}

The most interesting fibers are those for which   the morphism $\gamma\in \mathbb{T}M$ has the form $\gamma = (X_m,0)$, where $X_m$ is a tangent vector on $M$, and most of this subsection will be devoted to studying them.  

\begin{definition}  We shall denote by $S_0(\mathbb{T}M)$ the vector space quotient 
\[
S_0(\mathbb{T}M) = S(\mathbb{T}M)  \big/ t \cdot S(\mathbb{T}M) .
\]
Note that the $A(\mathbb{T}M)$-module structure on $S(\mathbb{T}M)$ descends to an  $A_0(\mathbb{T}M)$-module structure on $S_0(\mathbb{T}M)$. 
\end{definition}

The  quotient space   $S_0(\mathbb{T}M)$ is a graded vector space with nonzero components in integer degrees $-\dim(M)$ and up.  Indeed it is  the associated graded space 
for the decreasing filtration of the smooth sections of $S {\boxtimes}S^*$ by scaling order (to be clear, we place in degree $p$ the images of the elements $\sigma_pt^{-p}$, or in other words the sections of scaling order at least $p$, modulo the sections of scaling order $p{+}1$ or more). 

There is an obvious isomorphism
 \[
S(\mathbb{T}M)\vert _\gamma \cong  S_0(\mathbb{T}M) \big / I_{X_m}\cdot S_0(\mathbb{T}M),
\]
where $I_{X_m}$ is the     kernel of $\varepsilon _{X_m}$ in $A_0(\mathbb{T}M)$.     We shall use this to compute $S(\mathbb{T}M)\vert _\gamma$.

  \begin{definition}
  \label{def-eval-map-for-S-zero}
 Let $m\in M$.  We shall denote by 
  \[
\varepsilon_m \colon S _0 (\mathbb{T}M) \longrightarrow \wedge^* T_mM
\]
the  \emph{evaluation map} at $m\in M$  defined by the formula
\[
\varepsilon_m \colon \sum \sigma_p t^{-p}
\longmapsto \sum [\sigma_{-d}(m,m)]_{d} , 
\]
where  $[\,\underbar{\phantom{x}}\,]_{d}$ denotes the image in the quotient 
$
 \clifford_d  (T_mM) \big / \clifford _{d-1}(T_mM)  
$ 
of an element in $\Cliff_d (T_mM)$, and we identify the quotient with $\wedge^d T_mM$ via the quantization map.
 \end{definition}

  \begin{lemma}
The evaluation map has the property that 
\[
\varepsilon_m (f \sigma) = \varepsilon_m(f) \varepsilon_m (\sigma)
\]
for all $f \in A_0(\mathbb{T}M)$ and all $\sigma\in S_0(\mathbb{T}M)$. \qed
\end{lemma} 

\begin{definition}
  \label{def-nabla-map-for-S-zero}
Let  $X$ be a vector field on $M$. Denote by 
\[
\boldsymbol{\nabla_X}\colon S _0 (\mathbb{T}M) \longrightarrow S _0 (\mathbb{T}M) 
\]
the linear operator determined by the formula
\[
\boldsymbol{\nabla_X}\colon \sum \sigma_p t^{-p} \longmapsto 
\sum \nabla_X \sigma_p t^{-(p-1)} .
\]
\end{definition}

\begin{lemma}
\label{lem-derivation-property-for-S}
The operator $\boldsymbol{\nabla_X}$ is  compatible with the derivation $\boldsymbol{X}$ of the coordinate algebra $A_0(\mathbb{T}M)$ in the sense that 
\[
\boldsymbol{\nabla_X} (f\sigma) = \boldsymbol{X}(f ) \sigma + f \boldsymbol{\nabla_X} (\sigma)
\]
for every $f\in A_0(\mathbb{T}M)$ and every $\sigma \in S_0(\mathbb{T}M)$. \qed
\end{lemma}

The operator $\boldsymbol{\nabla_X}$ has grading degree minus one, and is therefore locally nilpotent. So we  can form the exponential 
\[
 \exp(\boldsymbol{\nabla_X}) \colon S _0 (\mathbb{T}M) \longrightarrow S _0 (\mathbb{T}M) 
\]
using the power series. Inspired by the discussion in Subsection~\ref{subsec-vector-fields-def-space}, let us  now make the  following definition:

\begin{definition}
Let $X$ be a vector field on $M$ and let $m\in M$.  Denote by 
\[
\varepsilon_{X_m} \colon S _0 (\mathbb{T}M) \longrightarrow \wedge^* T_mM
\]
the map defined by the commuting diagram
\[
\xymatrix{
S_0(\mathbb{T}M) \ar[r]^{\varepsilon_{X_m}} \ar[d]_{\exp(\boldsymbol{\nabla_X})}& \wedge^* T_mM \ar@{=}[d]\\
S_0(\mathbb{T}M) \ar[r]_{\varepsilon_m} & \wedge^* T_mM .
}
\]
\end{definition}

\begin{lemma}
\label{lem-compatibility-with-module-action}
The morphism  $  \varepsilon_{X_m}$ depends only on   the tangent vector  $X_m$, and not on  the values  of the vector field $X$ at other points in $M$. Moreover $\varepsilon _{X_m}(f \sigma) = \varepsilon_{X_m}(f) \varepsilon_{X_m} (\sigma)$.
\end{lemma}

 \begin{proof} This follows from Lemma~\ref{lem-derivation-property-for-S} and the definitions.
 \end{proof}

 \begin{proposition}
\label{prop-zero-fibers}
Let $m\in M$ and let $X_m\in T_mM$.  The morphism 
\begin{equation*}
\label{eq-eval-at-tangent-vector}
\varepsilon _{X_m} \colon S(\mathbb{T}M)\longrightarrow \wedge^* T_m M
\end{equation*}
induces an isomorphism 
\[
S(\mathbb{T}M)\vert _{(X_m,0)}\stackrel{\cong}\longrightarrow  \wedge^* T_mM.
\]
\end{proposition}

The proof will use the following local form for    sections of $S{\boxtimes}S^*$ of scaling order $p$.  Let $n=\dim(M)$, choose  a local orthonormal frame $\{\, e_1,\dots, e_n\,\}$ for $TM$ near $m\in M$, and 
for each index  
\[
I = ( i_1 < \cdots < i_d)
\]
of length $\ell(I) =d$, form the local section
\[
e_I = e_{i_1}\cdot e_{i_2}\cdot  \ldots \cdot e_{i_d}
\]
of the bundle $\clifford_{\ell(I)} (TM)$  over $M$.  View this as a local section of the restriction of $S{\boxtimes} S^*$   to the diagonal,  and use Theorem~\ref{prop-extension-from-diagonal} to extend it to a section over $M{\times} M$ with  scaling order $-\ell(I)$ near $m$.  We shall use the same notation $e_I$ for the extension.  
The smooth sections $e_I$ constitute a local frame for $S{\boxtimes} S^*$ near $(m,m)\in M{\times}M$.   So any smooth section $\sigma$ may be expanded in the form 
\begin{equation}
\label{eq-local-frame-expansion}
\sigma = \sum _I h_I \cdot e_I
\end{equation}
near $(m,m)$, where the $h_I$ are smooth, complex-valued  functions  on $M{\times}M$.  It follows from Lemma~\ref{lem-S-module-over-A} and the definition of scaling order that:

\begin{lemma}
\label{lem-local-frame-expansion}
Let $p\in \mathbb{Z}$.
The section $\sigma$ in \eqref{eq-local-frame-expansion} has scaling order $p$ near $m$ if and only if each scalar function $h_I$ vanishes on the diagonal of $M{\times}M$ near $(m,m)$ to order $p+ \ell(I)$, or more. \qed
\end{lemma}

\begin{proof}[Proof of Proposition~\ref{prop-zero-fibers}]
Lemma~\ref{lem-compatibility-with-module-action} shows that $\varepsilon_{X_m} $ does at least induce a vector space morphism 
\[
S(\mathbb{T}M)\vert_{(X_m,0)} \longrightarrow \wedge ^* TM.
\]
In addition, if $I=(i_1 < \cdots < i_d)$, then 
\begin{equation}
\label{eq-basic-sections-under-evaluation}
S(\mathbb{T}M)\ni e_I t^{\ell(I)} \longmapsto e_{i_1}\wedge \cdots \wedge e_{i_d} + \text{higher-degree  terms}  \in \wedge^* T_mM 
\end{equation}
It follows from this that $\varepsilon_{X_m}$ is surjective.

Now suppose that  $\sum \sigma_p t^{-p} \in S(\mathbb{T}M)$ is mapped to zero by $\varepsilon_{X_m}$.   We need to prove that it lies in $I_\gamma \cdot S (\mathbb{T}M)$. In doing so, we can assume that each $\sigma_p$ is supported near $(m,m)$. Indeed, 
if $\varphi$ is any smooth function on $M{\times}M$ that is equal to $1$ near $(m,m)$, then we can write 
\[
\sum \sigma_p t^{-p} = (1-\varphi) \cdot \sum \sigma_p t^{-p} + \sum \varphi \sigma_p t^{-p} ,
\]
and the first term on the right-hand side belongs to $I_{(X_m,0)}\cdot S(\mathbb{T}M)$ and so is mapped to zero by $\varepsilon_{X_m}$.  So we can replace each $\sigma_p$ by $\varphi \sigma_p$.

Assuming then that each    $\sigma_p$ is supported near $(m,m)$, we can write 
\[
\sigma_p = \sum_I h_{p,I} e_I ,
\]
as in \eqref{eq-local-frame-expansion}.  According to Lemma~\ref{lem-local-frame-expansion}, each  $h_{p,I}$ vanishes to order  $p {+} \ell(I)$ or higher on the diagonal in $M{\times}M$. Hence  we may write 
 \begin{equation}
 \label{eq-basic-sections-span1}
 \sum_p \sigma_p t^{-p}  
 	  = 
	\sum _{I} \,  \bigl (\,  \sum _p h_{p,I} t^{-(p+\ell(I))} \,\bigr )  \cdot  \bigl(  e_I t^{\ell(I)}  \bigr)  
 \end{equation}
where each  $\sum _p h_{p,I} t^{-(p+\ell(I))}$  is an element of (the complexification of) $A(\mathbb{T}M)$. 
To prove the proposition it suffices to show that if $ \sum \sigma_p t^{-p} $ maps to zero under the morphism $\varepsilon_{X_m}$ in the statement of the proposition, then     each function 
\[
\sum _p h_{p,I} t^{-(p+\ell(I))}\in  A(\mathbb{T}M)
\]
  evaluates to zero   at $X_m$.  But according to  \eqref{eq-basic-sections-under-evaluation}, the elements $e_It^{\ell(I)}\in S(\mathbb{T}M)$ map to linearly independent elements under $\varepsilon_{X_m}$.  So the required vanishing follows from Lemma~\ref{lem-compatibility-with-module-action}.
\end{proof}

\begin{proposition}
\label{prop-nonzero-fibers}
Let $ m_1, m_2\in M$ and $\lambda \in \R^{\times}$.   The morphism 
\[
\varepsilon_{(m_1,m_2,\lambda)} \colon S_0(\mathbb{T}M) \longrightarrow S_{m_1}\otimes S_{m_2}^*
\]
defined by the formula 
\begin{equation*}
\varepsilon_{(m_1,m_2,\lambda)} \colon  \sum \sigma_p t^{-p} \longmapsto \sum  \lambda^{-p}\sigma_p(m_1,m_2) 
\end{equation*}
induces an isomorphism
\[
S(\mathbb{T}M)\vert _{(m_1,m_2,\lambda )}\stackrel{\cong}\longrightarrow S_{m_1}^{\phantom{*}}\otimes S_{m_2}^*.
\]
\end{proposition}

\begin{proof}
The   morphism $\varepsilon_{(m_1,m_2,\lambda)}  $ above is obviously surjective and factors through the fiber, so it remains to prove injectivity of the induced map on the fiber.

Suppose first that $m_1=m=m_2$. Using the argument and  notation of the previous proof, if $\lambda \ne 0$, and if
\[
\sum_p \sum_I \lambda ^{-p} h_{p,I}(m,m) e_I (m,m) = 0  ,
\]
then for each $I$ 
\[
\sum _p \lambda ^{-(p+\ell(I))} h_{p,I} (m,m) = 0 ,
\]
since the vectors $\lambda ^{-\ell(I)}e_I (m,m)$ are linearly independent. So the formula \eqref{eq-basic-sections-span1} expresses any element of $S(\mathbb{T}M)$ that maps to zero in $S_m^{\phantom{*}}\otimes S_m^*$ as a combination  of elements in $A(\mathbb{T}M)$ that vanish at $(m,m,\lambda)$, times  elements in $S(\mathbb{T}M)$, as required.

If $m_1\ne m_2$, then we need only replace the local frame $\{ e_I\}$ near $(m,m)$ by any local frame of $S{\boxtimes} S^*$ near $(m_1,m_2)$ (away from the diagonal there is no need to invoke Proposition~\ref{prop-extension-from-diagonal}).  Then we may proceed as above.
\end{proof}

\subsection{The Rescaled Spinor Bundle}

We are now ready to construct the \emph{rescaled spinor bundle} $\mathbb{S}$ over $\mathbb{T}M$.

\begin{definition}
Define a family of vector spaces $\mathbb{S}_\gamma$ parametrized by  $\gamma \in \mathbb{T}M$ as follows:
 \begin{equation}
 \label{eq-fibers-rescaled-bundle}
 \mathbb{S}_\gamma 
 	= \begin{cases}
	S_{m_1}^{\phantom{*}}{\otimes}S_{m_2}^{*} & \gamma = (m_1,m_2,\lambda) \\
	\wedge^* T_mM & \gamma = (X_m,0) .
	\end{cases}
 \end{equation}
Denote by $\sigma\mapsto \widehat \sigma$
the morphism of $A(\mathbb{T}M)$-modules 
\begin{equation*}
 S(\mathbb{T}M) \longrightarrow  \prod_{\gamma\in \mathbb{T}M} \mathbb{S}_\gamma ,
\end{equation*}
that associates to each $\sigma \in S(\mathbb{T}M)$ its value in each fiber $S(\mathbb{T}M)\vert _\gamma$ under the identifications in  Propositions \ref{prop-zero-fibers} and \ref{prop-nonzero-fibers}.  
\end{definition}

\begin{lemma}
The above morphism   is  injective.
\end{lemma}
 
 \begin{proof}
 Let  $(m_1,m_2)\in M{\times}M$.  If an element $\sum_{p} \sigma_p t^{-p}$ maps to zero, then it follows from the formula for the morphisms $\varepsilon_{(m_1,m_2,\lambda)}$ that 
 \[
  \sum_{p}  \lambda^{-p}\sigma_p(m_1,m_2) = 0
\]
for all $\lambda \ne 0$.  But this implies that $\sigma_p(m_1,m_2)=0$ for all $p$.  Hence  $\sigma_p=0$ for all $p$, and so $\sum_{p} \sigma_p t^{-p} =0$.
 \end{proof}

 \begin{definition}
 \label{def-sheaf-of-sections}
 We shall denote by $\sheafS_{\mathbb{T}M}$ the sheaf on $\mathbb{T}M$ consisting of sections  
 \[
 \mathbb{T}M \ni \gamma \longmapsto \tau(\gamma) \in \mathbb{S}_\gamma
 \]
that are locally of the form
  \[
 \tau(\gamma)  = \sum_{j=1}^N f_j(\gamma)  \cdot \widehat \sigma_j(\gamma)
 \]
for some $N\in \N$,  where   $f_1,\dots, f_N$ are smooth, complex-valued functions on $\mathbb{T}M$  and   $\sigma_1,\dots, \sigma_N$  belong to $S(\mathbb{T}M)$.
 \end{definition}

\begin{theorem}
The sheaf $\sheafS_{\mathbb{T}M}$ is   locally free, of rank $2^{\dim(M)}$, as a sheaf of modules over  $\sheafA_{\mathbb{T}M}$.
\end{theorem}

\begin{proof} Let us prove that the sheaf is   free in a neighborhood of $\gamma = (X_m,0)$; the other    $\gamma\in \mathbb{T}M$ are handled in the same way.  Consider a section $\tau$ of $\sheafS_{\mathbb{T}M}$ as in Definition~\ref{def-sheaf-of-sections} above. Locally we may write each $\sigma_j\in S(\mathbb{T}M)$ as 
\[
\sigma_j = \sum _{I} \,  \bigl (\,  \sum _p h_{j,p,I} t^{-(p+\ell(I))} \,\bigr )  \cdot  \bigl(  e_I t^{\ell(I)}  \bigr)  
\]
as in \eqref{eq-basic-sections-span1}.  So near $\gamma$, the section $\tau$ is a linear combination  the sections $\widehat{e_I t^{\ell(I)}}$.  But these spanning sections are also linearly independent in each fiber $\mathbb{S}_\eta $, for $\eta$ near $\gamma$.  So they are linearly independent over the smooth functions on $\mathbb{T}M$, near $\gamma$, as required.
\end{proof}

 \begin{definition}
 We shall denote by $\mathbb{S}$ the unique smooth vector bundle over $\mathbb{T}M$ whose fibers are the spaces $\mathbb{S}_\gamma $ in \eqref{eq-fibers-rescaled-bundle} and whose smooth sections are the sections of the sheaf  $\sheafS_{\mathbb{T}M}$.
 \end{definition}
   
\subsection{Families of Differential Operators and the Getzler Symbol}
We wish to prove the following spinorial counterpart of Theorem~\ref{thm-families-of-ops}:

\begin{theorem}
\label{thm-getzler-family}
Let $D$ be a linear partial differential operator on $M$, acting on the sections of the spinor bundle, of Getzler order $q$.
The family of linear partial differential operators 
\[
D_{(m,\lambda)} = \lambda ^q D ,
\]
defined on those source fibers of $\mathbb{T}M$ with $\lambda \ne 0$, extends to a smooth family of linear partial differential operators  on all the source fibers of the tangent groupoid, acting on sections of $\mathbb {S}$. 
\end{theorem}

As with scalar case, it suffices to consider generators of the algebra of linear partial differential operators, in this case covariant derivatives $\nabla_X$ and Clifford multiplication operators  $c(X)$.  Let us begin with the latter, which are easier.  

\begin{lemma}
\label{lem-getzler-symbol-of-clifford-mult}
Let $X$ be a vector field on $M$.
The family of Clifford multiplication operators 
\[
D_{(m,\lambda)} = \lambda  c(X) ,
\]
defined on those source fibers of $\mathbb{T}M$ with $\lambda \ne 0$, and acting on sections of $\mathbb {S}$, extends to a smooth family on all the source fibers of $\mathbb{T}M$.  The operator on the source fiber  $\mathbb{T}M_{(m,0)} \cong T_{m}M$ is the exterior multiplication operator
\[
\wedge^* T_m M \ni \omega \longmapsto X_m\wedge \omega \in \wedge^* T_m M .
\]
\end{lemma}

\begin{proof}
The formula 
\[
\sum \sigma_p t^{-p} \longmapsto \sum c(X) \sigma_p t^{-(p-1)}
\]
defines a linear operator $\boldsymbol{c(X)}$ on $S(\mathbb{T}M)$. Moreover it is clear that 
\[
\varepsilon_{(m_1,m_2,\lambda)} \bigl (\boldsymbol{c(X)} \sigma \bigr ) = \lambda  c(X_{m_1})\cdot  \varepsilon_{(m_1,m_2,\lambda)}  ( \sigma  ) .
\]
This is enough to show that the family of operators $\{ \lambda c(X) \}$ on the source fibers with $\lambda \ne 0$ maps smooth sections of $\mathbb{S}$ over the full tangent groupoid to smooth sections. To compute the extension to the fibers where $\lambda = 0$ we need to show that for any tangent vector $Y_m$, 
\[
\varepsilon_{Y_m} (\boldsymbol{c(X)} \sigma \bigr ) = X_{m}\wedge   \varepsilon_{Y_m}  ( \sigma  ) .
\]
To see this, we note first that 
\[
\varepsilon_m (\boldsymbol{c(X)} \tau \bigr ) =
X_m\wedge \varepsilon_m (\tau \bigr )
\]
for any $\tau \in S_0(\mathbb{T}M)$, which is clear from the definitions.  Next,
the formula  
\[
\nabla_Y  c(X) - c(X) \nabla_Y  = c(\nabla_YX) \colon C^\infty ( M{\times}M, S{\boxtimes} S^*) \longrightarrow 
C^\infty ( M{\times}M, S{\boxtimes} S^*) 
\]
shows that the Getzler order-one operators $c(X)$ and  $\nabla_Y$ commute up to an operator of  Getzler order one, not two.  As a result,
\[
\boldsymbol{\nabla_Y} \boldsymbol{c(X)} = 
\boldsymbol{c(X)} \boldsymbol{\nabla_Y} \colon S_0(\mathbb{T}M) \longrightarrow S_0(\mathbb{T}M),
\]
and it  therefore follows that 
\[
\begin{aligned}
\varepsilon_{Y_m} \bigl ( \boldsymbol{c(X)} \sigma \bigr ) 
& =
\varepsilon_m \bigl (\exp(\boldsymbol{\nabla_Y})
\boldsymbol{c(X)} \sigma \bigr )  \\ 
& =
\varepsilon_m \bigl (
\boldsymbol{c(X)}  \exp(\boldsymbol{\nabla_Y})\sigma \bigr )  \\
& =
X_m \wedge \varepsilon_{Y_m}(\sigma),
\end{aligned}
\]
as required.
\end{proof}

 \begin{lemma}
\label{lem-getzler-symbol-of-nabla}
Let $X$ be a vector field on $M$.
The family of covariant derivatives 
\[
D_{(m,\lambda)} = \lambda  \nabla_X ,
\]
defined on those source fibers of $\mathbb{T}M$ with $\lambda \ne 0$, and acting on sections of $\mathbb {S}$, extends to a smooth family on all the source fibers of $\mathbb{T}M$.  The operator on the source fiber  $\mathbb{T}M_{(m,0)} \cong T_{m}M$ is the sum of directional differentiation in the direction $X_m$ and exterior multiplication by the linear function 
\[
T_mM \ni Y_m \longmapsto   \tfrac 12 \kappa (Y_m, X_m)\in \wedge^* T_mM,
\]
where, as in \eqref{eq-K-from-gamma}, the section  $\kappa(Y,X)$ of $\wedge^2 TM$ is related to the Riemannian curvature and the curvature of $S$ by 
\[
\gamma (R(Y,X)) = \kappa(Y,X)
\quad \text{and} \quad 
c ( q (\kappa (Y,X)) )= K(Y,X).
\]
\end{lemma}

The proof will use  following simple algebraic fact: 
\begin{lemma}
\label{lem-heisenberg}
If $A$, $B$ and $[A,B]$ are locally nilpotent linear operators on a rational vector space, and if $[A,B]$ commutes with both $A$ and $B$, then $A+B$ is locally nilpotent and 
\[
\pushQED{\qed} 
\exp(A)\exp(B) = \exp(  \tfrac 12 [A,B]) \exp(A+B)  . 
\qedhere
\popQED
\]
\end{lemma}

\begin{proof}[Proof of Lemma~\ref{lem-getzler-symbol-of-nabla}]
It follows from the definition of the curvature operator that 
\[
\left [{\nabla_Y}, {\nabla_X}\right ]  -{\nabla_{[Y,X]} } =  {K(Y,X) }
\]
as operators on smooth sections of $S{\boxtimes} S^*$ over $M{\times}M$.
So if we define 
\[
\boldsymbol{K (Y,X)} \colon 
S_0(\mathbb{T}M) \longrightarrow S_0(\mathbb{T}M)
\]
by
\[
\sum \sigma_p t^{-p} \longmapsto \sum K (Y,X) \sigma_p t^{-(p-2)} ,
\]
then, since $\nabla_{[Y,X]}$ has Getzler order one, not two, we obtain
\begin{equation*} 
\left [ \boldsymbol{\nabla_Y}, \boldsymbol{\nabla_X}\right ]  =  \boldsymbol{K (Y,X)} \colon S_0(\mathbb{T}M) \longrightarrow S_0(\mathbb{T}M) .
\end{equation*}
Moreover, as in the proof of Lemma~\ref{lem-getzler-symbol-of-clifford-mult}, each of $\boldsymbol{\nabla_Y}$ and $\boldsymbol{\nabla_X}$ commutes with $ \boldsymbol{K (Y,X)}$, and so by Lemma~\ref{lem-heisenberg}, 
\[
 \exp \bigl ( \boldsymbol{\nabla_{Y}} \bigr ) \exp \bigl ( \boldsymbol{\nabla_{X}} \bigr )
=
\exp\bigl ( \boldsymbol{\nabla_{Y+X}} \bigr )\exp\bigl ( \tfrac 12 \boldsymbol{K(Y,X)}\bigr ) .
\]
We can now compute that 
\[
\begin{aligned}
\varepsilon _{Y} \bigl (  \boldsymbol{\nabla_X} s \bigr ) 
	& = \frac{d}{dt}\Big \vert _{t =0} \varepsilon _{Y} \bigl (  \exp( \boldsymbol{\nabla_{tX}} )s \bigr ) \\
	& = \frac{d}{dt}\Big \vert _{t =0} \varepsilon_0\bigl ( \exp( \boldsymbol{\nabla_{Y}} )\exp( \boldsymbol{\nabla_{tX}} )s \bigr ) \\
& =  \frac{d}{dt}\Big \vert _{t =0}\varepsilon_0\bigl ( \exp( \boldsymbol{\nabla_{Y+tX}} )\exp( \tfrac 12 \boldsymbol {K(Y,tX)})s \bigr ) \\
& =  \frac{d}{dt}\Big \vert _{t =0}\varepsilon_{Y+tX}\bigl ( \exp( \tfrac 12 \boldsymbol {K(Y,tX)})s \bigr ) \\
& =  \frac{d}{dt}\Big \vert _{t =0}\varepsilon_{Y+tX}\bigl (s \bigr ) +  
\tfrac 12 \kappa (Y_m,X_m) \wedge \varepsilon_{Y} (s )  .\\
\end{aligned}
\]
In the last line we used the Leibniz rule and Lemma~\ref{lem-getzler-symbol-of-clifford-mult}.  We have now computed the action of the family $\{ \lambda \nabla_X\}$ in the statement of  the lemma  on ``algebraic'' sections  of $\mathbb{S}$ (associated to elements of $S(\mathbb{T}M)$).  The lemma follows from this.
\end{proof}

\subsection{Tangent Vectors Versus Normal Vectors}

So far, when discussing the tangent groupoid we have been identifying $TM$ with the normal bundle for the diagonal  in  $M{\times}M$  by associating to a tangent vector $X_m$ at $m\in M$ the  tangent vector $(X_m,0_m)$ at the diagonal point $(m,m)\in M{\times}M$.  In this subsection we shall examine the effect of doing otherwise.

Let $X$ be a vector field on $M$. Instead of writing $\nabla_X$ for the covariant derivative on $M{\times}M$ associated to the action on the left copy of $M$, let us temporarily write $\nabla_{(X,0)}$.  Let us similarly write $c(X,0)$ for left Clifford multiplication.  There are also obvious right operators $\nabla_{(0,X)}$ and $c(0,X)$, and let us begin by noting that all the right operators commute with the all the left operators.  

It follows from this commutativity that each right operator decreases the scaling order of a section $\sigma$ of $S{\boxtimes}S^*$   by at most one.  Consider for example the covariant derivative $\nabla_{(0,X)}$. If the scaling order of $\sigma$ is at least $p$, and if  $D$ is a differential operator of Getlzer order $q$ or less on $M$, acting on the left factor of $M{\times}M$, then we need to show that 
\[
\operatorname{Clifford-order}(D \nabla_{(0,X)} \sigma)
\le q - p + 1 .
\]
Write $\tau = D\sigma$, which is a section of scaling order $p{-}q$ or more. 
Since 
$D \nabla_{(0,X)} \sigma = \nabla_{(0,X)} D\sigma$,  we need to  show that 
\[
\operatorname{Clifford-order}(\nabla_{(0,X)} \tau)
\le q - p + 1 .
\]
 Next write 
\[
\nabla_{(0,X)} = \nabla_{(X,X)} - \nabla_{(X,0)} .
\]
The operator $ \nabla_{(X,X)} $ preserves Clifford order since along the diagonal of $M{\times}M$ the Riemannian connection is the standard connection on $\Cliff (TM)$, while of course $\nabla_{(X,0)}$ increases the Clifford order of $\tau$ by at most one, by definition of the scaling filtration.  The proof is complete. The proof for Clifford multiplications is similar, but simpler since the last step above is not needed.

It follows from these computations that we can define the scaling order using either left operators, or the right operators, or both. 

Now let  $X$ and $Y$ be vector fields on $M$.  Since  $\nabla_{(X,Y)}$ decreases scaling order by at most one, there is an induced, degree minus one  operator 
\[
\boldsymbol{\nabla_{(X,Y)}}
\colon S_0 (\mathbb{T}M) \longrightarrow S_0 (\mathbb{T}M)
\]
given by the now-usual formula 
\[
\sum \sigma_p t^{-p} \longmapsto \sum \nabla _{(X,Y)} \sigma_p t^{-(p-1)}
\]
on $S(\mathbb{T}M)$. Define the evaluation morphism 
\[
\varepsilon_{(X_m,Y_m)}
\colon S_0(\mathbb{T}M) \longrightarrow \wedge^* T_mM
\]
via the commuting diagram 
\[
\xymatrix@C=50pt{
S_0(\mathbb{T}M) \ar[r]^{\varepsilon_{(X_m,Y_m)}} 
\ar[d]_{\exp(\boldsymbol{\nabla_{(X,Y)}})} & \wedge^* T_mM 
\ar@{=}[d]
\\
S_0(\mathbb{T}M) \ar[r]_{\varepsilon_{m} }& \wedge^* T_mM .
}
\]
In the context of  $A_0(\mathbb{T}M)$, the map $\varepsilon_{(X_m,Y_m)}$   only depends on the normal vector determined by $(X_m,Y_m)$, but the following computation shows that this is not the case for  $S_0(\mathbb{T}M)$. Write $\kappa(X,Y) = \gamma (R(X,Y))$, as in Lemma~\ref{lem-getzler-symbol-of-nabla}.  View $\kappa(X,Y)$ as an operator on the exterior algebra bundle by exterior multiplication.

\begin{proposition}
\label{prop-normal-versus-tangent}
The diagram
\[
\xymatrix@C=50pt{
S_0(\mathbb{T}M) \ar[r]^{\varepsilon_{(X_m,Y_m)}} \ar@{=}[d] _{\phantom{xxxxxxxxxxxxx}}& \wedge^* T_mM 
\ar[d]^{ \exp( \frac 12 \kappa (X_m,Y_m))}
\\
S_0(\mathbb{T}M) \ar[r]_{\varepsilon_{(X_m-Y_m,0)}} & \wedge^* T_mM
}
\]
is commutative.
\end{proposition}

\begin{proof}
Let us first prove the following  special case: 
if $Y$ is any vector field on $M$, then 
\begin{equation}
\label{eq-normal-versus-tangent-special-case}
\varepsilon_{(0_m,Y_m)} = \varepsilon_{(-Y_m,  0_m)}\colon S_0(\mathbb{T}M) \longrightarrow \wedge^* T_mM .
\end{equation}
To do so, use the fact that ${\nabla_{(0,Y)}}$ and ${\nabla_{(-Y,0)}}$ commute as operators on smooth sections of $S{\boxtimes}S^*$ to write  
\[
\nabla_{(0,Y)}^n - \nabla_{(-Y,0)}^n 
=
 {\nabla_{(Y,Y)}}
 \Bigl ( 
\nabla_{(0,Y)}^{n-1} 
+  \nabla_{(0,Y)}^{n-2} \nabla_{(-Y,0)} 
+  \cdots 
+  \nabla_{(0,Y)}   \nabla_{(-Y,0)} ^{n-2}
+  \nabla_{(-Y,0)}^{n-1}
\Bigr ) .
\]
The operator $\nabla_{(Y,Y)}$ does not increase the Clifford order of sections, so if $\sigma $ is a section of scaling order $p$, then the section 
\[
\nabla_{(0,Y)}^n\sigma - \nabla_{(-Y,0)}^n\sigma
\]
has Clifford order $p{+}n{-}1$ or less. It follows now from the definitions that 
\[
\varepsilon_m \bigl (\boldsymbol{\nabla_{(0,Y)}}\sigma \bigr ) = \varepsilon _m \bigl (\boldsymbol{\nabla_{(-Y,0)}}\sigma \bigr),
\]
and \eqref{eq-normal-versus-tangent-special-case}  follows.

For the general case, it follows from Lemma~\ref{lem-heisenberg}   that 
\begin{equation} 
\label{eq-covariant-commutator}
 \exp (\boldsymbol{\nabla_{(-Y,0)}})
 \exp (\boldsymbol{\nabla_{(X,0)}})  =  \exp (\tfrac 12 \boldsymbol{K (X,Y)})  \exp (\boldsymbol{\nabla_{(X-Y,0)}})
\end{equation}
as operators on $S_0(\mathbb{T}M)$. Therefore
\[
\begin{aligned}
\varepsilon_{(X_m,Y_m)} ( \sigma ) 
	& = 
\varepsilon_m \bigl ( \exp (\boldsymbol{\nabla_{(X,Y)}})  \sigma \bigr ) \\
	& = 
\varepsilon_m \bigl ( \exp (\boldsymbol{\nabla_{(0,Y)}})
\exp (\boldsymbol{\nabla_{(X,0)}})  \sigma \bigr ) 
\\
& = 
\varepsilon_m \bigl ( \exp (\boldsymbol{\nabla_{(-Y,0)}})
\exp (\boldsymbol{\nabla_{(X,0)}})  \sigma \bigr ) 
\\
& = 
\varepsilon_m\bigl ( \exp (\tfrac 12 \boldsymbol{K (X,Y)})  \exp (\boldsymbol{\nabla_{(X-Y,0)}})\sigma \bigr ) \\
& = 
\exp (\tfrac 12 {\kappa (X_m,Y_m)})
\varepsilon_{X_m-Y_m}(\sigma ) 
\end{aligned}
\]
as required.
\end{proof}

\section{Multiplicative Structure}
\label{sec-multiplicative}

In this section   we shall equip the space $C_c^\infty (\mathbb{T}M, \mathbb{S})$  of smooth, compactly supported sections of the rescaled spinor bundle over $\mathbb{T}M$ with a convolution product. 

\subsection{Convolution Algebras of Smooth Groupoids}

We being by reviewing  some basic facts about the convolution algebras of smooth groupoids \cite[Section 2.5]{Connes94}. Let $s,t\colon \mathbb{G} \rightrightarrows M$ be a Lie groupoid. In order to build a convolution algebra of functions on $\mathbb{G}$ we shall fix a suitable family of measures on the target fibers of $\mathbb{G}$ (an alternative approach uses half-densities, but in our tangent groupoid  example  the family of measures has an extremely simple form). 

\begin{definition}[Compare {\cite[Section 1.2]{Renault80}}]
A \emph{smooth left Haar system} on $\mathbb{G}$  is a family   of smooth measures $\mu^m$  on the target fibers 
\[
\mathbb{G}^m = \{ \, \gamma \in \mathbb {G} : t(\gamma) = m\,\}
\]
of  $\mathbb{G}$   having the following two properties:
\begin{enumerate}[\rm (i)]
\item For any compactly supported smooth function $f$ on $\mathbb{G}$, the assignment
$$
m \mapsto \int_{\mathbb{G}^m}  f(\gamma) \, d\mu^m (\gamma) 
$$
defines a smooth function on $M$.
\item For any morphism $\gamma_1: m\to p$ and   any compactly supported smooth function $f$ on $\mathbb{G}$ we have
$$
\int_{\mathbb{G}^m}  f(\gamma_1  \circ \gamma ) \, d\mu^m(\gamma) = \int_{\mathbb{G}^p} f(\gamma) \, d\mu^p(\gamma)
$$
\end{enumerate}
\end{definition}

Given a smooth left Haar system on $\mathbb{G}$, the formula
\[
f_1\star f_2 (\eta) = \int _{\mathbb{G}^{t (\eta)}} f_1 (\gamma) f_2 (\gamma^{-1}\circ \eta)\, d\mu^{t (\eta)} (\gamma) ,
\]
 defines an associative product on $C_c^\infty (\mathbb{G})$.  This is the convolution algebra of the Lie groupoid $\mathbb{G}$.

\subsection{Multiplicative Structures on Bundles Over Groupoids}
\label{subsec-multiplicative1}
Let $s,t\colon \mathbb{G} \rightrightarrows M$ be  a Lie groupoid, once again. Form the space of composable pairs 
\[
\mathbb{G}^{(2)} = \{\, (\gamma,\eta) \in  \mathbb{G}{\times}\mathbb{G} : s(\gamma) = t(\eta)\, \} ,
\]
and denote by 
\[
c \colon  \mathbb{G}^{(2)} \longrightarrow \mathbb{G}
\quad \text{and} \quad 
p_1,p_2\colon \mathbb{G}^{(2)} \longrightarrow \mathbb{G}
\]
the composition map $c (\gamma, \eta) = \gamma \circ \eta$  and the two coordinate projections.

\begin{definition}
Let $\mathbb{V}$ be a smooth vector bundle over  $\mathbb{G}$. A \emph{multiplicative structure} on $\mathbb{V}$ is a morphism of vector bundles
\[
p_{1}^* \, \mathbb{V} \otimes p_{2}^* \, \mathbb{V} \stackrel \circ \longrightarrow c^* \, \mathbb{V} ,
\]
or in other words a smoothly varying family of vector space morphisms
\[
\mathbb{V}_\eta \otimes \mathbb{V}_\gamma  \stackrel \circ\longrightarrow \mathbb{V} _{\eta\circ \gamma} ,
\]
that is associative in the natural sense that 
\begin{equation*}
v_\alpha  \in \mathbb{V}_\alpha, \quad v_\beta  \in \mathbb{V}_\beta, \quad v_\gamma  \in \mathbb{V}_\gamma 
\quad \Rightarrow \quad 
 v_\alpha \circ (v_\beta \circ v_\gamma) 
 =  (v_\alpha \circ v_\beta) \circ v_\gamma \in \mathbb{V}_{\alpha\circ \beta \circ \gamma}
\end{equation*}
for all composable $\alpha$, $\beta$ and $\gamma$.
\end{definition}

\begin{example}
\label{ex-standard-structure}
If $ \mathbb{G} \rightrightarrows M$ is any smooth groupoid, and if $V$ is a vector bundle on $M$ then the bundle on $\mathbb{G}$ with fibers
\[
\mathbb{V}_\gamma  =    V_{t(\gamma)}^{\phantom{*}}\otimes  V_{s(\gamma)}^{*} = \operatorname{Hom} \bigl (V_{s(\gamma)},V_{t(\gamma)}\bigr )
\]
 has an obvious multiplicative structure given by contraction/composition  that we shall call the \emph{standard multiplicative structure}.
\end{example}

\begin{lemma}
\label{lem-muliplicative-convolution}
Let $\mathbb{G}$ be a Lie groupoid equipped with a smooth left Haar system.  If $\mathbb{V}$ is a vector bundle on $\mathbb{G}$ with multiplicative structure, then the formula 
\[
f_1\star f_2 (\eta) = \int _{\mathbb{G}^{t (\eta)}} f_1 (\gamma)\circ  f_2 (\gamma^{-1}\circ \eta)\, d\mu^{t (\eta)} (\gamma) ,
\]
defines an associative product on the   smooth, compactly supported sections of $\mathbb{V}$. \qed
\end{lemma}

\subsection{Multiplicative Structure on the Rescaled Spinor Bundle}

The rescaled spinor bundle $\mathbb{S}$ over the tangent groupoid that we constructed in Section~\ref{sec-rescaled} carries the standard multiplicative structure away from $\lambda=0$ since 
\[
\mathbb{S}\vert _{\lambda\ne 0}  = S\boxtimes S^* .
\]
The purpose of this section is to prove  the following result: 
\begin{theorem}
\label{thm-multiplicative-structure-on-S}
There is a unique multiplicative structure on the rescaled spinor bundle $\mathbb{S}$ over $\mathbb{TM}$ whose restriction away from $\lambda =0$ is the standard multiplicative structure.  On fibers at $\lambda =0$ the multiplication map
\[
\mathbb{S}_{(X_m,0)}\otimes \mathbb{S}_{(Y_m,0)}\longrightarrow \mathbb{S}_{(X_m+Y_m,0)} 
\]
is given by the formula
\begin{equation*}
\alpha \otimes \beta \longmapsto \alpha \wedge \beta \wedge 
	\exp \bigl(- \tfrac 12 \kappa (X_m,Y_m)\bigr ) .
\end{equation*}
\end{theorem}

The uniqueness statement in the theorem is clear since $\mathbb{T}M\setminus TM{\times}\{0\}$ is dense in $\mathbb{T}M$.  To prove the existence statement we shall show that  if $\rho,\tau \in S(\mathbb{T}M)$, and if the associated sections
of $\mathbb{S}$ are pulled back to $\mathbb{T}M^{(2)}$ via $p_1$ and $p_2$, and then  multiplied according to the formula in the statement of the theorem, then the result, namely
\begin{equation}
\label{eq-composition1}
\begin{cases}  (m_1,m_2,m_3,\lambda) \longmapsto \varepsilon_{(m_1,m_2,\lambda)}(\rho)\circ \varepsilon_{(m_2,m_3,\lambda)}(\tau) 
& (m_1,m_2,m_3\in M\quad \lambda {\ne} 0)
\\
(X_m,Y_m,0)\longmapsto   \varepsilon _{X_m}(\rho) \wedge \varepsilon _{Y_m} (\tau)  \wedge  \exp\bigl (\tfrac 12 \kappa (Y_m,X_m)\bigr )
& (X_m,Y_m\in T_mM )
\end{cases}
\end{equation}
is a smooth section of the pullback bundle  $c^*\mathbb{S}$ over $\mathbb{T}M^{(2)} $.  This will suffice. By linearity it further suffices to consider elements $\rho,\tau\in S(\mathbb{T}M)$ of the form
\[
\rho = \rho_{p_1} t^{-p_1} \quad \text{and} \quad \tau = \tau_{p_2} t^{-p_2} ,
\]
where $\rho_{p_1}$ and $\tau_{p_2}$ have scaling orders at least $p_1$ and $p_2$, respectively. 

Form the pointwise composition
\[
M{\times}M{\times}M \ni (m_1,m_2,m_3) \longmapsto  \rho_{p_1}(m_1,m_2)\circ \tau_{p_2}(m_2,m_3)\in S_{m_1}\otimes S_{m_3}^* ,
\]
which is a smooth section of the pullback to $M{\times}M{\times}M$ of $S{\boxtimes}S^*$ along the projection onto the first and third factors (which is the composition map for the pair groupoid).  As we did in Subsection~\ref{subsec-rescaled-module},  choose a local frame $\{e_I\}$ of $S{\boxtimes} S^*$ consisting of sections whose scaling orders are at least the negatives of their Clifford orders.  We can of course write 
\begin{equation}
\label{eq-composition2}
\rho_{p_1}(m_1,m_2)\circ \tau_{p_2}(m_2,m_3) = 
\sum _{I} f_I(m_1,m_2,m_3) \cdot e_I(m_1,m_3)
\end{equation}
where each $f_I$ is a smooth function on (an open subset of) $M{\times}M{\times}M$.

\begin{lemma}
\label{lem-pullback-scaling-order}
Each function $f_I$ defined above vanishes to order $p_1{+}p_2{+}\ell(I)$ or more on the diagonal $M\subseteq M{\times}M{\times}M$.
\end{lemma}

\begin{proof}
Let us call an index $I$ \emph{regular} if $f_I$ vanishes to order $p_1{+}p_2{+}\ell(I)$ or more on the diagonal, and \emph{deficient} otherwise.  Write 
\begin{multline}
\label{eq-composition-1}
\rho_{p_1}(m_1,m_2)\circ \tau_{p_2}(m_2,m_3) - \sum_{I\,\text{regular}}f_I(m_1,m_2,m_3) \cdot e_I(m_1,m_3) 
\\= 
\sum _{I\,\text{deficient}} f_I(m_1,m_2,m_3) \cdot e_I(m_1,m_3) .
\end{multline}
The left-hand side has    scaling order $p_1+p_2$ or more  in the sense Definition~\ref{def-scaling-order}, except using covariant derivatives and Clifford multiplications in both the first and third factors in $M{\times}M{\times}M$.  If there were any deficient indices at all, then we could choose a deficient $I_{\mathrm{min}}$ for which the vanishing order of $f_I$  was minimal.  Call the vanishing order $q$;   of course
\begin{equation}
\label{eq-deficient-ineq}
q< p_1+p_2+\ell(I_{\mathrm{min}})
\end{equation}
by definition of deficiency.   We could then find a differential operator $D$ of order  $q$ so that $D( f_I e_I)$ is a smooth function multiple of  $e_I$ along the diagonal for all deficient $I$, and a nonzero function multiple for $I_{\mathrm{min}}$.  But the the Clifford order of the right-hand side of \eqref{eq-composition-1} after applying $D$  would be at  least $\ell(I_{\mathrm{min}})$, whereas the Clifford order of the left-hand side after applying $D$ would be  at most $q-p_1-p_2$. This contradicts \eqref{eq-deficient-ineq}. \end{proof}

Now write 
\[
F_I = f_I t^{-(p_1+p_2 + \ell(I))} \in A (\mathbb{T}M^{(2)})\quad \text{and} \quad 
\sigma_I = e_I t^{\ell(I)} \in S(\mathbb{T}M) .
\]
We should like to prove that the section \eqref{eq-composition1} is given by the formula
\begin{equation}
\label{eq-composition3}
\begin{cases}
 (m_1,m_2,m_3,\lambda) \longmapsto \sum_{I} \varepsilon_{(m_1,m_2,m_3,\lambda)} (F_I)\varepsilon_{(m_1,m_3,\lambda)}(\sigma) 
& 
\\
(X_m,Y_m,0)\longmapsto    \sum_{I} \varepsilon_{(X_m,0_m,-Y_m)} (F_I) \varepsilon _{X_m+Y_m}(\sigma_I) . & 
\end{cases}
\end{equation}
See \eqref{eq-composable-tangent-groupoid-pairs} for the notation. Since \eqref{eq-composition3} is a combination of smooth functions on $\mathbb{T}M^{(2)}$, times pullbacks to $\mathbb{T}M^{(2)}$ of smooth sections of $\mathbb{S}$, this will suffice.  

The identity of \eqref{eq-composition1} and \eqref{eq-composition3} away from $\lambda=0$ is clear, and we have seen in Proposition~\ref{prop-normal-versus-tangent} that  
\[
 \varepsilon _{(X_m+Y_m, 0_m)}(\sigma_I) = \tfrac 12 \kappa (Y_m,X_m)\wedge \varepsilon_{(X_m,-Y_m)}(\sigma_I).
\]
So in fact it suffices to prove that 
\begin{equation*}
\varepsilon _{X_m}(\rho) \wedge \varepsilon _{Y_m} (\tau)  = \sum_{I} \varepsilon_{(X_m,0_m,-Y_m)} (F_I) \varepsilon _{(X_m, -Y_m)}(\sigma_I) ,
\end{equation*}
or, using \eqref{eq-normal-versus-tangent-special-case}, that 
\begin{equation}
\label{eq-composition4}
\varepsilon _{(X_m,0_m)}(\rho) \wedge \varepsilon _{(0_m,-Y_m)} (\tau)  = \sum_{I} \varepsilon_{(X_m,0_m,-Y_m)} (F_I) \varepsilon _{(X_m, -Y_m)}(\sigma_I) .
\end{equation}
A systematic way to check this formula is to introduce the space $S(\mathbb{T}M^{(2)})$ of Laurent polynomials $\sum \sigma_p t^{-p}$ in which $\sigma_p$ is a smooth section of 
$S{\boxtimes} \C{\boxtimes} S^*$ over $M{\times} M {\times} M$ that has scaling order $p$, as in the proof of Lemma~\ref{lem-pullback-scaling-order}. This is a module over $A(\mathbb{T}M^{(2)})$, and we have the obvious identity
\[
\rho \circ \tau  = \sum _I F_I \cdot \sigma_I
\]
in $S(\mathbb{T}M^{(2)})$, and hence  in the quotient 
\[
S_0(\mathbb{T}M^{(2)}) =
S(\mathbb{T}M^{(2)}) \big / t\cdot S(\mathbb{T}M^{(2)}).
\]
The morphism
\[
\varepsilon _m \colon S_0(\mathbb{T}M^{(2)}) \longrightarrow \wedge^* T_mM
\]
defined, following Definition~\ref{def-eval-map-for-S-zero}, by 
\[
\varepsilon _m 
\colon \sum \sigma_p t^{-p} \longmapsto 
\sum [\sigma_{-d} (m,m,m)]_d
\]
has the properties that 
\[
\varepsilon _m(\rho\circ \tau) =  \varepsilon _m(\rho) \wedge \varepsilon_m(\tau)
\]
and that 
\[
\varepsilon _m(F\cdot \sigma)  =  \varepsilon _m(F) \cdot  \varepsilon_m(\sigma) ,
\]
and these settle \eqref{eq-composition4} in the special case where $X_m=Y_m=0$.  The general case is settled by applying the special case to the elements
\begin{multline*}
\overline{\rho} = \exp\bigl ( \boldsymbol{\nabla_{(X,0)}} \bigr ) \rho ,
\quad 
\overline{\tau}= \exp\bigl ( \boldsymbol{\nabla_{(0,-Y)}}\bigr ) \tau,
\\
\quad \overline{F_I} =  \exp\bigl ( \boldsymbol{\nabla_{(X,0,-Y)}} \bigr ) F_I ,
\quad \text{and} \quad 
\overline{\sigma_I} =  \exp\bigl ( \boldsymbol{\nabla_{(X,-Y)}}\bigr ) \sigma_I ,
\end{multline*}
for which $\overline{\rho}\circ\overline{\tau} = \sum_I \overline{F_I}\cdot \overline{\sigma_I}$.

\section{Convolution Algebra and Traces}
\label{sec-convolution-algebra}

The multiplicative structure on $\mathbb{S}$ provides us with a convolution algebra $C^\infty_c(\mathbb{T}M,\mathbb{S})$. In this section we shall construct our family of supertraces on this algebra.

\subsection{A Haar System for theTangent  Groupoid}

Let $M$ be a smooth manifold.  The  target fibers of $\mathbb{T}M$ are of course 
\[
\mathbb{T}M^{(m,\lambda)} = \{m\} {\times} M {\times} \{ \lambda \}
\]
and 
\[
\mathbb{T}M^{(X_m,0)} =T_m M  {\times} \{ 0\}
\]
If we fix a smooth measure $\mu$ on $M$, and if we denote by $\mu^m$ the associated translation-invariant measures on the tangent spaces $T_mM$, then the formulas 
\begin{equation}
\label{eq-left-Haar}
\left \{ 
\begin{aligned} 
\mu^{(m,\lambda) } & = |\lambda|^{-n} \mu \\
\mu^{(X_m,0) } & =   \mu^m
\end{aligned}
\right .
\end{equation}
define a smooth left Haar system for $\mathbb{T}M$.  So we can now form the associated \emph{tangent groupoid  algebra} $C_c^\infty(\mathbb{T}M)$.

 \begin{proposition}
 \label{prop-sections-tangent-groupoid-alg}
For $\lambda \ne 0$ the linear map
\[
\varepsilon _\lambda \colon C_c^\infty (\mathbb{T}M) \longrightarrow \mathfrak{K} (L^2 (M))
\]
given by the formula 
\[
 \varepsilon _\lambda(f) \colon (m_1,m_2) \longmapsto \lambda^{-n} f(m_1,m_2,\lambda)
 \]
is a homomorphism of algebras. In addition the linear  map 
\[
\varepsilon _0 \colon C_c^\infty (\mathbb{T}M, \mathbb{S}) \longrightarrow  C_c^\infty (TM)
\]
given by the formula 
\[
 \varepsilon _0(f) \colon X_m \longmapsto  f(X_m,0)
 \]
is a homomorphism of algebras, too, if the target $C_c^\infty (TM)$ is equipped with the fiberwise convolution product.  \qed
\end{proposition}

\subsection{Twisted Convolution on the Tangent Bundle}
 
The first statement in Proposition~\ref{prop-sections-tangent-groupoid-alg} has an obvious spinorial counterpart: 

 \begin{proposition}
For $\lambda \ne 0$ the morphism
\[
\varepsilon _\lambda \colon C_c^\infty (\mathbb{T}M, \mathbb{S}) \longrightarrow \mathfrak{K} ^\infty (L^2 (M,S))
\]
given by the formula 
\[
 \varepsilon _\lambda(\sigma) \colon (m_1,m_2) \longmapsto \lambda^{-n} \sigma(m_1,m_2,\lambda)
 \]
is a homomorphism of algebras. \qed
\end{proposition}

We can also define 
 \[
\varepsilon _0 \colon C_c^\infty (\mathbb{T}M, \mathbb{S}) \longrightarrow  C_c^\infty (TM, \wedge^*TM) 
\]
by   restriction, so that 
\[
\varepsilon _0(\sigma ) \colon X_m \longmapsto  \sigma (X_m,0)  .
\]
But in order to make this a homomorphism of algebras we need to adjust the convolution operation on $C_c^\infty (TM, \wedge^*TM) $, in accordance with Theorem~\ref{thm-multiplicative-structure-on-S}, as follows.
 
 \begin{proposition}
If the space $C_c^\infty (TM, \wedge^* TM)$ is equipped with the twisted convolution product 
\[
(\varphi_1\star \varphi_2 )(X_m) = \int _{T_mM} \varphi _1(X_m - Y_m) \wedge \varphi_2(Y_m)\wedge  \exp (\tfrac 12  \kappa (Y_m, X_m) ) \, d \mu^m(Y_m) ,
\]
then the restriction map 
\[
\varepsilon _0 \colon C_c^\infty (\mathbb{T}M, \mathbb{S}) \longrightarrow  C_c^\infty (TM, \wedge^*TM)
\]
is a homomorphism of algebras. \qed
\end{proposition}

\begin{remark} This is an appropriate time to note that our approach shares much with a manuscript of  Siegel \cite{Siegel10}. The above proposition is, however at variance with the corresponding formula there \cite[p.16]{Siegel10}.
\end{remark}

\subsection{Supertraces on the Clifford Algebra}

Let  $E$ be an even-dimensional and oriented Euclidean vector space.  Let  $e_1,\dots, e_n$ be an oriented orthonormal basis for $E$, and for   $I = (i_1< i_2 < \cdots <i_d)$ 
let 
\[
e_I = e_{i_1}\cdot e_{i_2} \cdot \ldots \cdot e_{i_d} \in \clifford (E) .
\]
The linear functional 
\[
\str\colon \clifford (E) \longrightarrow \C
\]
defined by 
\begin{equation}
\label{eq-str-fmla2}
\str(e_I ) = \begin{cases} 1 & I=(1,2,\dots, n) \\ 0 & \text{otherwise.}
\end{cases}
\end{equation}
is independent of the choice of oriented, orthonormal basis, and is a supertrace on the Clifford algebra. 
 See \cite[Sec.\ 2.2.8]{Meinrenken13}.
Note that 
\begin{equation}
\label{eq-str-fmla3}
\str\vert _{\clifford_{n-1}(E)} = 0 .
\end{equation}

The supertrace can be calculated using the  irreducible representation
\[
 c\colon \clifford (E) \stackrel \cong \longrightarrow \End(S)
 \]
 as follows.  
 The element 
\begin{equation}
\label{eq-grading-matrix}
s = i ^{\frac n2} e_1\cdot \ldots \cdot e_n \in \clifford (E)
\end{equation} 
is independent of the choice of  {oriented} orthonormal basis and  satisfies $s^2 = 1$. The self-adjoint operator $c(s)$  determines a $\Z/2$-grading of the vector space $S$, and 
\begin{equation}
\label{eq-str-fmla4}
\str(x) =    \left ( \tfrac i 2 \right ) ^{\frac n 2} \Tr (c(s) c(x))
\end{equation}
for all $x\in \clifford(E)$.

\subsection{Supertraces on the Convolution Algebra}
\label{subsec-supertraces}

The Hilbert space $L^2 (M,S)$ carries a $\mathbb Z/2$-grading that is defined as follows. If $e_1,\dots, e_n$ is any local oriented orthonormal frame for the tangent bundle of $M$, then the product 
\[
c(s) = i ^{\frac n2} c(e_1) c(e_2) \dots c(e_n)  
\]
defines locally an endomorphism of $S$ whose square is the identity.  It is in fact independent of the choice of local oriented orthonormal frame, and so the formula above defines a canonical global endomorphism of $S$. It is self-adjoint and  squares to the identity, and is by definition the grading operator for the $\Z/2$-grading on $L^2 (M,S)$.

We shall denote by 
\[
\Str\colon  \mathfrak{K}^\infty(L^2 (M,S))  \longrightarrow \C
\]
the associated supertrace, and we shall use this to define a family of supertraces on $C_c^\infty (\mathbb{T}M, \mathbb{S})$, as follows:

\begin{definition} 
For $\lambda \in \R \setminus \{ 0\}$ we shall denote by 
\[
\Str_\lambda\colon C_c^\infty (\mathbb{TM}, \mathbb{S}) \longrightarrow \C
\]
 the composition
\[
C_c^\infty (\mathbb{TM}, \mathbb{S})\stackrel{\varepsilon _\lambda}\longrightarrow   \mathfrak{K}^\infty(L^2 (M,S)) \stackrel{\Str} \longrightarrow \C ,
\]
as in \eqref{eq-Str-lambda-neq-0}.
In addition, we define the supertrace $\Str_0$ using \eqref{eq-Str-lambda-eq-0}.
\end{definition}

\begin{theorem}
\label{thm-smoothly-varying-traces}
If $\tau \in C_c^\infty (\mathbb{T}M, \mathbb{S})$, then $\lambda\mapsto \Str_\lambda (\tau)$ is a smooth function of $\lambda \in \R$.
\end{theorem}

\begin{remark}
The ordinary tangent groupoid algebra carries a family of traces, param\-etrized by $\lambda \ne 0$, that are obtained by composing the homomorphisms \eqref{eq-eval-t-ne-0} with the usual operator trace on smoothing operators:
\[
C_c^\infty (\mathbb{T}M) \stackrel{\varepsilon_\lambda} \longrightarrow \mathfrak{K}^\infty (L^2 (M)) \stackrel{\Tr} \longrightarrow \C.
\]
Roughly speaking, local, or algebraic, index theory  is   the study of these traces as  $\lambda \to 0$.  The traces   do not converge as $\lambda\to0$, and  instead more  elaborate strategies must be developed, for instance   replacing the traces with equivalent cyclic cocycles. See  for example  \cite{NestTsygan95} or \cite{Perrot13} for two perspectives on this.  It is a remarkable fact, discovered of course by Getzler, that in the supersymmetric context the traces \emph{do} converge.
\end{remark}

\begin{proof}[Proof of Theorem~\ref{thm-smoothly-varying-traces}]
The supertrace on $\mathcal{K}^\infty (L^2 (M,S))$ can be written 
\[
\Str(k) = \int _M  \str  \bigl (k(m,m)\bigr )\, d \mu (m),
\]
where  $\str$ is the pointwise supertrace on 
$\End(S_m) $.
So according to the definitions, if $\tau \in C_c^\infty (\mathbb{T}M,\mathbb{S})$ and $\lambda \ne 0$, then 
\[
 \Str_\lambda (\tau )   = \lambda ^{-n} \int _M \str \bigl (\tau(m,m,\lambda)\bigr )\, d \mu (m) .
 \]
We shall show that for any smooth section $\tau$  the map
\[
(m,\lambda )\longmapsto  \lambda ^{-n}\str \bigl (\tau(m,m,\lambda)\bigr )
\]
extends to a smooth function on $M{\times}\R$, and then calculate   the value of the extension at $0\in \R$  to be  
\begin{equation}
\label{eq-value-at-0}
(m,0 )\longmapsto   \str (\tau(0_m,0))
,
\end{equation}
where $0_m\in T_mM $ is the zero tangent vector  and the supertrace is the coefficient of $e_1\wedge \cdots \wedge e_n \in \wedge^* T_mM$, with $e_1,\dots, e_n$ as above. This will suffice.

Any smooth section of $\mathbb{S}$ over $\mathbb{T}M$ is locally a finite sum of products 
$f \cdot \widehat \sigma$,   where $f$ is a smooth function on $\mathbb{T}M$ and $\sigma\in S(\mathbb{T}M)$; see Definition~\ref{def-sheaf-of-sections}.  Since 
\[
 \lambda ^{-n}\str \bigl ((f\cdot \widehat \sigma)(m,m,\lambda)  \bigr )
 =
 f(m,m,\lambda)\cdot \lambda^{-n}\str \bigl (\widehat \sigma(m,m,\lambda)  \bigr )
 \]
 it suffices to show that $\lambda^{-n}\str   (\widehat \sigma(m,m,\lambda)    )$ extends to a smooth function on $M{\times}\R$, and calculate that the value of the extension  at $\lambda = 0$ agrees with \eqref{eq-value-at-0}.
 
 If $\sigma = \sum \sigma_p t^{-p}$, then 
 \[
 \lambda^{-n}\str \bigl (\widehat \sigma(m,m,\lambda)  \bigr )
 =
 \sum \lambda^{-p-n}\sigma_p(m,m)
 \]
 Now   if $p> -n$, then the restriction of $\sigma_p$ to the diagonal point $(m,m)$  lies in 
 \[
 c\bigl ( \Cliff_{n-1}(T_mM)\bigr )\subseteq \End(S_m),
 \]
  and hence by \eqref{eq-str-fmla3} it has supertrace   zero.  So after writing $q = -p$ we find that   
 \[
 \lambda^{-n}\str \bigl (\widehat \sigma(m,m,\lambda)  \bigr )
 =
 \sum_{q \ge n}  \lambda^{q-n}\str\bigl ( \widehat \sigma_{-q}(m,m)\bigr ) ,
 \]
 which is clearly a smooth function of $m\in M$ and $\lambda \in \R$.  The value at $\lambda=0$ is $\str  ( \widehat \sigma_{-q}(m,m)  )$, and if we write 
 \[
 \sigma_{-n} = \sum_I h_I e_I
 \]
 as in \eqref{eq-local-frame-expansion}, then from    \eqref{eq-str-fmla2} we find that 
  \[
 \str\bigl ( \widehat \sigma_{-n}(m,m)\bigr ) = h_{I_n}(m,m) 
 \]
where  $I_n = (1,2,\dots, n)$.  This is the coefficient of $e_1\wedge\cdots \wedge e_n$ in the fiber $\wedge^* T_mM$, as required.
\end{proof}

\subsection{Final Comments on Index Theory}

In this concluding subsection we shall comment on the roles that the tangent groupoid and rescaling play in index theory, and suggest future developments,   which we aim to pursue elsewhere. 

Let us return to Theorem~\ref{thm-families-of-ops}.  We noted there that the family of operators $\{ D_{(m,\lambda)}\}$ on the source fibers of the tangent groupoid that is associated to a single linear partial differential operator on $M$ is \emph{equivariant} for the (right)  action of the groupoid $\mathbb{T}M$ on itself.  This has the following consequence:

\begin{lemma}
\label{lem-module-map}
The family of operators $\{ D_{(m,\lambda)}\}$ on the source fibers of $\mathbb{T}M$ acts on the function space $C_c^\infty (\mathbb{T}M)$ as a right $C^\infty_c(\mathbb{T}M)$-module endomorphism. 
\end{lemma}

If  $D$ is in addition elliptic, then we can say   more.  To make the cleanest statement it is convenient to introduce the quotient algebra 
$C_c^\infty (\mathbb{T}M)_{[0,1]}$ by the ideal of all smooth, compactly supported functions on $\mathbb{T}M$ that vanish   for all $\lambda \in [0,1]$.  Of course  the family $\{D_{(m,\lambda)}\}$ acts on this algebra by right module endomorphisms, too. Ellipticity implies that this action is almost invertible: 

\begin{theorem}
\label{thm-parametrix}
If $M$ is closed, and if $D$ is elliptic, then the associated right-module endomorphism of   $C_c^\infty (\mathbb{T}M)_{[0,1]}$ is invertible modulo left multiplications by elements of $C_c^\infty (\mathbb{T}M)_{[0,1]}$.
\end{theorem}

To be explicit, the theorem asserts that  there are right module maps 
\[
\mathbb{D}\colon C_c^\infty (\mathbb{T}M)_{[0,1]} \longrightarrow C_c^\infty (\mathbb{T}M)_{[0,1]}
\quad \text{and} \quad 
\mathbb{Q}\colon C_c^\infty (\mathbb{T}M)_{[0,1]} \longrightarrow C_c^\infty (\mathbb{T}M)_{[0,1]} ,
\]
the first associated to $\{ D_{(m,\lambda)}\}$, for which the operators 
\[
\mathbb{I} -\mathbb{D}\mathbb{Q} \colon C_c^\infty (\mathbb{T}M)_{[0,1]} \longrightarrow C_c^\infty (\mathbb{T}M)_{[0,1]}
\quad \text{and} \quad 
\mathbb{I} -\mathbb{Q}\mathbb{D} \colon C_c^\infty (\mathbb{T}M)_{[0,1]} \longrightarrow C_c^\infty (\mathbb{T}M)_{[0,1]}
\]
are  left multiplications by elements of $C_c^\infty (\mathbb{T}M)_{[0,1]}$.  
The theorem may be proved using pseudodifferential operator theory  (and see \cite{VanErpYuncken15} for an account of the theory of pseudodifferential operators that is particularly well suited to the present context).  

The theorem implies that $\mathbb{D}$    defines a class in  $K_0(C_c^\infty (\mathbb{T}M)_{[0,1]} )$;  see for example \cite[Sec.\ 2]{Milnor71}.  This is  an essential step in Connes' approach to index theory via $K$-theory and the tangent groupoid. 

\begin{remarks} 
Actually when considering $K$-theory it is preferable to pass to a Fr\'echet algebra completion of $C_c^\infty (\mathbb{T}M)_{[0,1]} $, as in  $\cite{CarrilloRouse08} $, or, even better, the $C^*$-algebra completion considered by Connes in  \cite[Sec.\ II.5]{Connes94}.  In addition,    in order to get a sufficiently rich class of examples, one should introduce operators acting on sections of bundles, and use the associated modified convolution algebras, as in Example~\ref{ex-standard-structure} and Lemma~\ref{lem-muliplicative-convolution} above.
\end{remarks}

It is an interesting  challenge to fit  the rescaled bundle and the algebra $C_c^\infty (\mathbb{T}M,\mathbb{S})$ into this type of  $K$-theory picture. The main issue is that the Dirac operator $\slashed{D}$ gives rise to a family of operators for which the analogue of  Lemma~\ref{lem-module-map} holds, but \emph{not} the analogue of Theorem~\ref{thm-parametrix}, the latter because the model operators $\slashed{D}_{(m,0)}$ are not elliptic, as they are in the standard case (as is well known they are in fact the de Rham differentials on the tangent fibers).    Perhaps Kasparov's Dirac operator $d_M$ from \cite[Def.\ 4.2]{Kasparov88} has a role to play here. 

There are other interesting  challenges, too.  For instance although the convolution algebra  $C_c^\infty (\mathbb{T}M,\mathbb{S})$ admits natural  Fr\'echet and Banach algebra completions \cite{ZelinYi19}, there is no $C^*$-algebra completion. 
  
 Getzler took a different approach that focussed not on  $\slashed{D}$ but on the Laplace-type operator $\Delta = \slashed {D}{}^2$, for which the model operators $\Delta_{(m,0)}$ are variants of the quantum harmonic oscillator (and are elliptic).  Supersymmetry relates the supertraces considered in Subsection~\ref{subsec-supertraces} to the index of the Dirac operator:
    
\begin{lemma}
\label{lem-supersymmetry}
The supertrace $\Str  \bigl (\exp (-\lambda^2 \Delta)\bigr )$ is the index of the Dirac operator $\slashed{D}$, and is in particular a constant, integer-valued  function of $\lambda\ne 0$.
\end{lemma}
 
 As Getzler pointed out, the smoothness of the family of supertraces $\Str_\lambda$ from Subsection~\ref{subsec-supertraces}  now allows one to compute the index from the value at $\lambda {=}0$, which involves only the operators $\Delta_{(m,0)}$, which depend only on the Riemannian curvature of $M$. See \cite{Getzler83,BerlineGetzlerVergne92,Roe98}.  It will be interesting to explore this more thoroughly from the point of view of the cyclic cohomology of the algebra $C_c^\infty (\mathbb{T}M,\mathbb{S})$, and also discover what lessons can be   learned in $K$-theory and $K$-homology about the use of  $\slashed{D}{}^2$ rather than $\slashed{D}$ here.

\section{Appendix. Taylor Expansions}
\label{sec-taylor}

The purpose of this appendix is to prove Proposition~\ref{prop-extension-from-diagonal}.  We shall use the exponential map 
\[
TM \ni X_m \longmapsto (\exp_m (X_m), m) \in  M{\times}M,
\]
which is a diffeomorphism from a neighborhood of the zero section in the tangent bundle onto a neighborhood of the diagonal in $M{\times}M$, and the associated Euler vector field $E$, defined on a neighborhood of the diagonal in  $M{\times}M$, by
\begin{equation*}
E_{(\exp(X_m), m)} = \frac{d}{ds}\Big \vert _{s = 1} (\exp_m (sX_m),m) .
\end{equation*}
The Euler vector field is tangent to each source fiber $M{\times}\{m\}$ of the pair group\-oid, and if $(x_1,\dots, x_n)$ are geodesic local coordinates on $M$ that are centered at $m$, then 
\begin{equation*}
E = \sum _{i=1}^n x_i \partial_i 
\end{equation*}
on $M{\times} \{ m\}$.  We shall also use the concept of Taylor series that is explained in the following two definitions.

 \begin{definition}
 We shall say that a smooth section $\sigma$ of $S{\boxtimes}S^*$ is \emph{synchronous near $m\in M$}, if $\nabla_E\sigma =0$ in a nieghborhood of  $(m,m)\in M{\times}M$.
\end{definition}

By parallel translation, every smooth section of $S\boxtimes S^*$   on the diagonal extends to a smooth section that is synchronous near the diagonal.

\begin{definition}
Let $m\in M$ and let  $(x_1,\dots, x_n)$ be smooth functions defined in  a neighborhood of $(m,m)\in M{\times}M$  that restrict to geodesic local coordinates at  $(m',m')$   on each $M{\times}\{m'\}$.
Let $\sigma$ be a  smooth section  $S{\boxtimes}S^*$.   A 
\emph{Taylor expansion}  of the section $\sigma $ at $m\in M$ is a formal series 
\begin{equation}
\label{eq-taylor-expansion}
 \sum _{\alpha \ge 0}   x^\alpha  \,\sigma_\alpha   ,
\end{equation}
where 
\begin{enumerate}[\rm (i)]

\item the sum is over multi-indices $\alpha=(\alpha_1,\dots, \alpha_n)$, with each $\alpha_k$ a nonnegative integer, and $x^\alpha= x_1^{\alpha_1}\cdots x_n^{\alpha_n}$;

\item each $\sigma_\alpha$ is   a smooth section of $S{\boxtimes}S^*$ that is synchronous near  $m$
(note that since it is synchronous near $m$, $\sigma_\alpha$   is determined by its values along the diagonal near $m$)
\item the series is asymptotic to $\sigma$ near the diagonal and near $(m,m)\in M{\times}M$  in the sense that for every $N\in \mathbb{N}$ the difference 
\[
\sigma - \sum _{|\alpha|<  N}   x^\alpha \,\sigma_\alpha
\]
vanishes to order $N$ on the diagonal near $(m,m)$ (here $|\alpha| = \alpha_1+ \cdots + \alpha_n$).
\end{enumerate}
\end{definition}

Every smooth section has a unique Taylor expansion.  Proposition~\ref{prop-extension-from-diagonal} is a consequence of the following result:

\begin{proposition} 
\label{prop-scaling-order-and-taylor}
Let $\sigma$ be a  smooth section of the   $S{\boxtimes}S^*$, and let $m\in M$. If 
\[
\sigma \sim  \sum _{\alpha \ge 0}   x^\alpha  \,\sigma_\alpha   
\]
is the Taylor series of $\sigma$ near $m$, then
\begin{equation}
\label{eq-scaling-vs-taylor-order}
\operatorname{Scaling-order} ( \sigma ) 
\ge \min _\alpha  \bigl \{ \, |\alpha| - \CliffordOrder(\sigma_\alpha)\, \bigr \} 
\end{equation}
near $m$. In particular, if $\sigma$ is synchronous near $m$, then
\[
\operatorname{Scaling-order} ( \sigma ) 
\ge -\CliffordOrder(\sigma) .
\]
near $m$.
  \end{proposition}

The proposition is proved as follows.  Let us temporarily call the quantity on the right hand side of \eqref{eq-scaling-vs-taylor-order} the \emph{Taylor order} of $\sigma$. Obviously
\[
\operatorname{Taylor-order}(\sigma) \le - \operatorname{Clifford-order}(\sigma)
\]
If we can prove that applying an operator $D$ to $\sigma$ decreases the Taylor order by at most the Getzler order of $D$, then we shall get 
\[
\begin{aligned}
\operatorname{Taylor-order}(\sigma) - \operatorname{Getzler-order}(D) 
	& \le 
\operatorname{Taylor-order}(D\sigma)  \\
& \le - \operatorname{Clifford-order}(D\sigma)
\end{aligned}
\]
and hence 
\[
\operatorname{Clifford-order}(D\sigma) \le \operatorname{Getzler-order}(D) -\operatorname{Taylor-order}(\sigma)
\]
In view of Definition~\ref{def-scaling-order},  the proposition follows immediately from this.  As for the effect on the Talyor order of applying $D$, it is clear that a Clifford multiplication $c(X)$ increases it by at most one; the other case to consider, that of a covariant derivative $\nabla_X$, is handled by the following lemma:

\begin{lemma} 
\label{lem-taylor}
Let $\sigma$ be a  smooth section of the   $S{\boxtimes}S^*$ that is synchronous near $m\in M$, and let $X$ be a vector field on $M$. The  Taylor series at $m$ of the section  $\nabla_X \, \sigma $  has the form 
  \[
\nabla_X \, \sigma  \sim  \sum _{|\alpha|\ge 1}   x^\alpha \, c(q(\omega_\alpha))  \,  \sigma   ,
\]
  where  each   $\omega_\alpha$ is the germ near $m\in M$ of a smooth section of $\wedge ^2 TM$. Here we regard   $c(q(\omega_\alpha))$ as a section of $S{\boxtimes}S^*$ defined on the diagonal near $m$, and extend it to a section over $M{\times}M$  that is synchronous near $m$.
  \end{lemma}

\begin{proof}
(Compare \cite[Prop.\ 12.22]{Roe98}.)
A general vector field $X$ on $M$ can be written as a combination $\sum_i f_i \partial_i$, and by expanding the smooth coefficient functions $f_i$ in Taylor series we see that it suffices to prove the lemma for the coordinate vector fields  $X= \partial_i$.

 According to the definition of curvature,
\begin{equation}
\label{eq-def-of-curvature}
\nabla_E \nabla _X \, \sigma  - \nabla_X \nabla _E \, \sigma  - \nabla _{[E,X]} \, \sigma   = K(E,X) \,  \sigma  .
\end{equation}
Since the section $ \sigma $ is synchronous near $m$,
\begin{equation}
\label{eq-euler-vector-field0}
\nabla _E  \sigma  = 0 
\end{equation}
in a neighborhood of $(m,m)\in M{\times}M$.  Moreover, since $X$ is a coordinate vector field, 
\begin{equation}
\label{eq-euler-vector-field1}
[E,X] = - X .
\end{equation}
Inserting  \eqref{eq-euler-vector-field0} and \eqref{eq-euler-vector-field1}  into   \eqref{eq-def-of-curvature}  we find that 
\begin{equation}
\label{eq-euler-vector-field2}
\nabla_E \nabla_X \, \sigma  + \nabla_X\,  \sigma   = K(E,X)\, \sigma  .
\end{equation}
Now expand $\nabla_X\,  \sigma $ as a Taylor series at $m\in $M,
\begin{equation}
\label{eq-euler-vector-field3}
\nabla _X  \sigma  \sim  \sum_{|\alpha|\ge 1}  x^\alpha \sigma _\alpha
\end{equation}
(there is no order zero term because the section $ \sigma $  is synchronous).
Using  the formula 
\[
\nabla _E \, x^\alpha \,  \sigma _\alpha = |\alpha|\, x^\alpha\, \sigma _\alpha
\]
for the Euler vector field we find  that the Taylor series for $\nabla_E \nabla_X \,  \sigma $ is 
\begin{equation}
\label{eq-euler-vector-field4}
\nabla_E \nabla_X \,  \sigma    \sim    \sum_{\alpha}  |\alpha|\, x^\alpha  \sigma _\alpha.
\end{equation}

Next,  recall that the curvature operator $K(E,X)$ may be written as 
\[
K(E,X) = c\bigl  (\gamma (R(E,X))\bigr ) .
\]
See \eqref{eq-K-from-gamma}.  Write the section $\gamma (R(E,X))$ of $\wedge^2 TM$  as a Taylor series
\begin{equation}
\label{eq-euler-vector-field5}
\gamma (R(E,X)) \sim  \sum_{|\alpha|\ge 1} x^\alpha \eta_\alpha ,
\end{equation}
where each  $\eta_\alpha \in \wedge^2 TM$ is synchronous at $m\in M$ for the Levi-Civita connection (there is no order zero term in this Taylor expansion either, this time because the vector field $E$ vanishes at $m\in  M$).
Inserting \eqref{eq-euler-vector-field3}, \eqref{eq-euler-vector-field4}  and \eqref{eq-euler-vector-field5} into  \eqref{eq-euler-vector-field2}  we obtain an identity of Taylor expansions 
\[
\sum_{|\alpha|\ge 1} (1+ |\alpha|)\, x^\alpha  \sigma _\alpha =   \sum_{|\alpha|\ge 1} x^\alpha c(q(\eta_\alpha))  \sigma  .
\]
The lemma follows from this.
\end{proof}

\bibliography{Refs} 
\bibliographystyle{alpha}

\noindent {\small  Department of Mathematics, Penn State University, University Park, PA 16802.}

\smallskip

\noindent{\small Email: higson@psu.edu and  zuy106@psu.edu.}

\end{document}